\documentclass[a4paper,10pt,reqno]{amsart}

\usepackage[T1]{fontenc}
\usepackage[utf8]{inputenc}
\usepackage{lmodern}
\usepackage[english]{babel}
\usepackage[a4paper,top=3cm,bottom=2.5cm,left=2.5cm,right=2.5cm]{geometry}

\usepackage{amsmath}
\usepackage{amssymb}
\usepackage{mathtools}
\usepackage{mathrsfs}
\usepackage[normalem]{ulem}
\usepackage[
colorlinks=true,
urlcolor=teal,
citecolor=magenta,
linkcolor=teal,
breaklinks=true]
{hyperref}

\usepackage[noabbrev,capitalize,nameinlink]{cleveref}
\usepackage{bookmark}

\usepackage[initials,msc-links,
]{amsrefs}

\usepackage[dvipsnames]{xcolor}
\usepackage{graphicx}
\usepackage{url}

\usepackage{enumitem}

\usepackage{aurical} 
\usepackage{changes}

\usepackage{tikz-cd}

\numberwithin{equation}{section}

\theoremstyle{plain}
\newtheorem{theorem}{Theorem}[section]
\newtheorem{lemma}[theorem]{Lemma}
\newtheorem{proposition}[theorem]{Proposition}
\newtheorem{corollary}[theorem]{Corollary}

\theoremstyle{definition}
\newtheorem{definition}[theorem]{Definition}
\newtheorem{example}[theorem]{Example}
\newtheorem{remark}[theorem]{Remark}

\definecolor{deepgreen}{rgb}{0.1,0.6,0.1}

\newcommand{\hh}{\mathbb H}

\newcommand{\diver}{{\mathrm{div}}}

\newcommand{\G}{\mathbb{G}}
\newcommand{\g}{\mathfrak g}
\newcommand{\f}{\mathfrak f}
\newcommand{\m}{\mathfrak m}
\newcommand{\vg}{\nu^\G}
\newcommand{\otdp}{\left(\Delta_S\right)_p}
\newcommand{\otdgt}{\left(\Delta_S\right)_{\gamma(t)}}
\newcommand{\otd}{\Delta_S}
\newcommand{\J}{\mathbf{J}}

\newcommand{\rr}{\mathbb R}

\renewcommand{\v}{\nu^{\mathbb G}}

\newcommand{\leb}{\mathcal{L}}
\newcommand{\lie}{\mathrm{Lie}}

\newcommand{\otc}[1]{\left(\Delta_S\right)_{#1}}

\newcommand{\N}{\mathbb{N}}

\newcommand{\R}{\mathbb{R}}

\newcommand{\spann}{\mathrm{span}}

\newcommand{\average}{{\mathchoice {\kern1ex\vcenter{\hrule height.4pt
				width 6pt
				depth0pt} \kern-9.7pt} {\kern1ex\vcenter{\hrule height.4pt width 4.3pt
				depth0pt}
			\kern-7pt} {} {} }}
\newcommand{\ave}{\average\int}

\newcommand{\Om}{\Omega}

\newcommand{\rank}{\mathrm{rank}}
\newcommand{\Span}{\mathrm{span}}
\newcommand{\Nil}{\mathrm{Nil}}
\newcommand{\nil}{\mathfrak{nil}}
\newcommand{\order}{\mathrm{order}}
\newcommand{\sk}{\mathrm{Skew}}
\newcommand{\id}{\mathrm{id}}

\DeclarePairedDelimiter{\set}{\{}{\}}

\DeclareMathOperator{\Lie}{Lie}
\allowdisplaybreaks

\newcommand{\df}{\mathrm{d}}
\DeclareMathOperator{\divv}{div}
\definechangesauthor[name=Nicola, color=brown]{N}
\definechangesauthor[name=Simone, color=purple]{S}
\definechangesauthor[name=Luca, color=blue]{L}
\definechangesauthor[name=Enrico, color=magenta]{E}

\begin{document}

\title[hypergenerated groups]{hypergenerated Carnot groups}

 \author[E. Le Donne]{Enrico Le Donne}
\address[E. Le Donne]{University of Fribourg, Chemin du Mus\'ee~23, 1700 Fribourg, Switzerland}
\email{\href{mailto:enrico.ledonne@unifr.ch}{enrico.ledonne@unifr.ch}}
	\author[L. Nalon]{Luca Nalon}
 \address[L. Nalon]{University of Fribourg, Chemin du Mus\'ee~23, 1700 Fribourg, Switzerland}
\email{\href{mailto:luca.nalon@unifr.ch}{luca.nalon@unifr.ch}}
	\author[N. Paddeu]{Nicola Paddeu}
 \address[N. Paddeu]{University of Fribourg, Chemin du Mus\'ee~23, 1700 Fribourg, Switzerland  
}
\email{\href{mailto:nicola.paddeu@unifr.ch}{nicola.paddeu@unifr.ch}}

\author[S.~Verzellesi]{Simone Verzellesi}
\address[S.~Verzellesi]{Dipartimento di Matematica "Tullio Levi-Civita", Università degli Studi di Padova, via Trieste 63, 35131 Padova (PD), Italy}
\email{\href{mailto:simone.verzellesi@unipd.it}{simone.verzellesi@unipd.it}}

\date{\today}



\thanks{\textit{Memberships and funding information.} E. Le Donne, L. Nalon, and N. Paddeu were partially supported by the Swiss National Science Foundation
	(grant 200021-204501 `\emph{Regularity of sub-Riemannian geodesics and applications}'). S. Verzellesi is member of the Istituto Nazionale di Alta Matematica (INdAM), Gruppo Nazionale per l'Analisi Matematica, la Probabilità e le loro Applicazioni (GNAMPA). S. Verzellesi is supported by MIUR-PRIN 2022 Project \emph{Regularity problems in sub-Riemannian structures},  project code 2022F4F2LH, and by the INdAM-GNAMPA 2025 Project \emph{Structure of sub-Riemannian hypersurfaces in Heisenberg groups}, CUP ES324001950001.}

\begin{abstract}
    In this paper we provide an algebraic characterization of those stratified groups in which boundaries with locally constant normal are locally flat. We show that these groups, which we call \emph{hypergenerated}, are exactly the stratified groups where embeddings of non-characteristic hypersurfaces are locally bi-Lipschitz. Finally, we extend these results to submanifolds of arbitrary codimension.
\end{abstract}

\bibliographystyle{abbrv}
\maketitle

\tableofcontents

\section{Introduction}
 
 In this paper 
 we systematically introduce a class of Carnot groups which we call hypergenerated. Namely, a Carnot group is \emph{hypergenerated} if all its vertical hyperplanes are Carnot subgroups. We refer to \cref{carnot_groups} for the definition of Carnot group and vertical hyperplane. Moreover, we refer to \cref{subsec:plentiful} for a comprehensive treatment of hypergenerated groups and Lie algebras. We show that the algebraic property of being hypergenerated captures features relevant for both the regularity of minimal boundaries and the metric properties of non-characteristic hypersurfaces. More precisely, we prove the following characterization.

\begin{theorem}
\label{thm:main-theorem}
 
     Let $\G$ be a Carnot group. The following are equivalent:
     \begin{enumerate}
         \item the group $\G$ is hypergenerated,
         \item all finite-perimeter sets with locally constant normal are vertical hyperplanes,
         \item on all smooth non-characteristic hypersurfaces in $\G$, the intrinsic distance and the restricted distance are locally bi-Lipschitz equivalent.
     \end{enumerate}
   
 \end{theorem}

The framework and results of \cref{thm:main-theorem} naturally extends to a hierarchy that we called \emph{hypergenerated of order $k$}, for $k \geq 1$, see \cref{subsec:plentiful}.

We refer to \cref{embedded_surfaces} for the notion of non-characteristic hypersurfaces, intrinsic distance and restricted distance, and to the statement of \Cref{hgiffchnimplhp} for the notion of constant normal. Condition $(2)$ deals with the regularity theory for minimal boundaries. The historical approach to this problem has evolved along different paths, fostering the development of the theories of \emph{currents} by Federer and Fleming (see \cite{MR123260}), of \emph{varifolds} by Almgren (see \cite{MR225243}) and Allard (see \cite{MR307015}), and of \emph{sets of finite perimeter} by De Giorgi (see \cite{MR179651}). However, an important meeting point of these three approaches consists in employing different variations of the so-called \emph{Lipschitz approximation theorem} (see e.g. \cite{MR652826}). In the language of finite perimeter sets, the latter states that minimal boundaries are close, in measure, to graphs of Lipschitz functions with arbitrary small Lipschitz constant. The distinction between the minimal boundary and the Lipschitz graph is quantified by the so-called \emph{excess}, a local quantity that measures the oscillation of the normal.  This modern viewpoint, for which we refer to the monograph \cite{MR2976521}, does not depend on the validity of perimeter monotonicity formulas.
Although the validity of a sub-Riemannian monotonicity formula is still a major open problem, the Lipschitz approximation theorem has been fully generalized in \cites{MR3194680,MR3682744} to the Heisenberg groups $\hh^n$, when $n\geq 2$, via approximation by \emph{intrinsic Lipschitz graphs}. What distinguishes higher-dimensional Heisenberg groups from the first Heisenberg group $\hh^1$, where only partial results are available, is the presence in the latter of sets with zero excess that are not locally flat, i.e., they are not locally vertical hyperplanes. To avoid these types of phenomena, the authors of \cite{psv} identified a family of stratified groups of step $2$, which they called \emph{plentiful groups}, in which sets with locally constant normal are locally flat, thus extending the Lipschitz approximation theorem to this family of groups. The same notion was recently studied from a control theoretic perspective in \cite{master-theses-Llorens}.
The equivalence $(1) \Leftrightarrow (2)$ in \Cref{thm:main-theorem} shows that hypergenerated groups are \emph{exactly} those stratified groups where the sets with locally constant normal are locally vertical hyperplanes. Thus, they represent the right environment for the development of the regularity theory in the setting of stratified groups. 

Motivated by the previous discussion, we develop the theory of hypergenerated groups from an algebraic viewpoint. We show that the property of being hypergenerated only depends on the maximal step $2$ quotient of the group, see \cref{quotient_hypergenerated}. Moreover, we provide a characterization of step $2$ hypergenerated Lie groups in terms of the image of the Kaplan's operator, see \cref{hyper_order}. This result reveal a correspondence between the notion of hypergenerated groups and a class of problems 
known as \emph{MinRank problems}. The latter are deeply studied from the computational complexity point of view, with applications in cryptanalysis (see \cites{complexity_minrank,cripto_minrank}). We briefly discuss this at the end of \cref{algebraic_characterization}.

With this algebraic machinery, we prove that the class of hypergenerated group is actually vast, containing stratified groups of arbitrarily large step. This result extends to hypergenerated groups of order $k$, for every $k \ge 1$, see \cref{existencehypergenerated}. Moreover, we classify all indecomposable hypergenerated groups of topological dimension up to $7$. 

The property of the group of being hypergenerated can be described purely in terms of distance functions. On a subset of a metric space, one can consider two distance functions: either the distance function of the ambient space restricted to the subspace or the induced length metric. On hypersurfaces inside the Euclidean space, these two distance functions are locally bi-Lipschitz equivalent. However, the latter statement in no longer true for Carnot groups. Indeed, \cref{thm:main-theorem} shows that the two distance functions are locally bi-Lipschitz equivalent on all non-characteristic hypersurfaces if and only if the Carnot group is hypergenerated. The bi-Lipschitz equivalence of restricted and intrinsic distance functions for smooth non-characteristic hypersurfaces so far was only known for Heisenberg groups (see \cite{Gioacchino-Enrico-Pauls}).
To prove \cref{thm:main-theorem}, we study weak tangents of contact maps. Following \cite{gioacchino-sebastiano-enrico}, we show that the at every point of an equiregular sub-Riemannian manifold the weak tangent is unique and is a Carnot group. In analogy with \cite{Margulis-Mostow}, we then show that the weak tangent of a contact map between equiregular sub-Riemannian manifolds is unique and is a group homomorphism. We apply this result to the inclusions of non-characteristic surfaces inside of hypergenerated Carnot groups to get the implication $(2)\Rightarrow(3)$ in \cref{thm:main-theorem}.

\vspace{2mm}

The paper is organized as follows. In \cref{sec:preliminaries} we recall some basic properties and terminology of Carnot groups and sub-Riemannian manifolds. In \cref{subsec:plentiful} we introduce and study hypergenerated groups of order $k$. In \cref{sec:locally_const_normal} we prove the implication $(1)\Leftrightarrow(2)$ of \cref{thm:main-theorem}, whereas \cref{sectionnonchar} is dedicated to the proof of the implication $(1)\Leftrightarrow(3)$.





\section{Preliminaries}\label{sec:preliminaries}

\subsection{Equiregular sub-Riemannian manifolds}

For an introduction to sub-Riemannian manifolds, we refer to \cites{Enrico2025book,Mongomery_book}.
A (constant rank) {\em distribution} $\Delta$ on $M$ is a smooth subbundle of the tangent bundle $TM$. In the following, if $M$ is a smooth manifold, we denote by $\Gamma(TM)$ the family of smooth vector fields defined over $M$, and by $\Gamma(\Delta)$ the family of vector fields in $\Gamma(TM)$ that are tangent to $\Delta$. 
For all $p\in M$, we define $\Delta^1_p:=\Delta_p$ and
\begin{equation}
	\Delta^k_p:=\mathrm{span}\{[X_{1}\ldots [X_{j-1},X_{j}]\ldots ](p)\mid j\leq k, \, X_i\in\Gamma(\Delta)\}, \qquad \text{for every }  k>1.
\end{equation}
A distribution $\Delta$ is {\em bracket-generating} if, for all $p\in M$, there exists $k$ such that $\Delta^k_p=T_pM$. We call $\dim(\Delta_p)$ the {\em rank} of $\Delta$ at $p$ and we say that a distribution $\Delta$ has rank $m$ if $m=\dim(\Delta_p)$ for all $p\in M$. We say that a distribution $\Delta$ is {\em equiregular} if the integer $\dim(\Delta^k_p)$ is constant in $p$ for every $k\in\N$.
A {\em sub-Riemannian manifold} is a connected manifold equipped with a bracket-generating distribution $\Delta$, with a Riemannian metric $g$ defined on $\Delta$, and with the \emph{Carnot-Carathéodory distance} 
\begin{equation}
\label{eq:def_dcc}
    \begin{split}
        d_{cc}(x,y):=\inf\Big\{\int_0^1\sqrt{g(\dot{\gamma}(t),\dot{\gamma}(t))}\df t \, \big|\,\gamma:[0,1]\to M \text{ absolutely continuous; }\\\dot{\gamma}(t)\in \Delta_{\gamma(t)}\text{ for a.e. } t\in[0,1];\, \gamma(0)=x;\, \gamma(1)=y \Big\}.
\end{split}
\end{equation}
Since $\Delta$ is bracket-generating, the \emph{Chow-Rashevskii connectivity theorem} (see \cite{MR1880} and \cite[Chapter 3]{Enrico2025book}) ensures that $d_{cc}$ is a finite distance. We say that a sub-Riemannian manifold with distribution $\Delta$ is \emph{equiregular} if $\Delta$ is equiregular.
\begin{definition}[Stratified basis]
\label{def:stratified_basis}
Let $M$ be an equiregular sub-Riemannian manifold with distribution $\Delta$ of rank $m$. Let $p\in M $ and $U\subseteq M$ be a neighborhood of $p$. We say that $X_1,\ldots ,X_n\in \Gamma(TU)$ is a \emph{stratified basis} for the distribution $\Delta$ at $U$ if:
\begin{enumerate}
    \item $X_1,\ldots ,X_{m}$ is an orthonormal frame for $\Delta$;
    \item for all $r\in\{1,\ldots,n-1\}$ and for all $j\in \{r+1,\ldots,n\}$ the vector $X_j$ is an iterated Lie bracket of $X_1,\ldots,X_{r}$;
    \item   if $i<j$, then $w_i\leq w_j$, where
    \begin{equation*}
    w_j:=\min\left\{k \ | \ X_j=[X_{i_1}\ldots [X_{i_{k-1}},X_{i_k}]\ldots ]\}, \ i_1,\ldots ,i_k\in \{1,\ldots ,m\}\right\}.
\end{equation*}
\end{enumerate}
We refer to the integer $w_j$ as the \emph{weight} of $X_j$.

\end{definition}

\subsection{Carnot groups} \label{carnot_groups}

For a thorough introduction on Carnot groups, we refer to \cites{MR2363343, primer,Enrico2025book, MR3587666}.
A \emph{stratified Lie algebra of step $s$}, with $s\in\mathbb{N}$, is a nilpotent Lie algebra $\g$, together with a direct sum decomposition
\begin{equation*}
\mathfrak{g}=V_1\oplus V_2\oplus\cdots\oplus V_s,
\end{equation*}
such that the vector spaces $V_1,\dots,V_s\subseteq\mathfrak g$ satisfy
\begin{equation*}
V_i=[V_1,V_{i-1}]\quad \text{for}\ i=1,\dots,s-1, \quad V_s \neq \set{0}, \quad [V_1,V_s]=\set{0}. 
\end{equation*}
We refer to such decomposition as a \emph{stratification of $\g$}, to $V_1,\dots,V_s$ as the \emph{strata} (or \emph{layers}) of the stratification, and to $m \coloneq \dim(V_1)$ as the \emph{rank} of $\g$. A \emph{stratified group of step $s$ and rank $m$} is a simply connected Lie group $\G$ whose associated Lie algebra $\g$ is stratified of step $s$ and rank $m$. In order to make sense of the notation used in the rest of the paper, we stress that the Lie algebra $\g$ associated to a Lie group $\G$ is understood as its tangent space at the identity element $\id \in \G$. If $p \in G$, we denote by $L_p \colon \G \to \G$ the \emph{left translation} by $p$, i.e., the diffeomorphism $q \mapsto p \cdot q$.
Once a basis $X_1,\dots,X_n$ of a Lie algebra $\g$ associated to a stratified group $\G$ is fixed, we can identify $\G$ with $\R^n$ via \emph{exponential coordinates of the first type}, i.e.,
\begin{equation*}
        \Phi \colon \R^n \to \G, \qquad (y_1,\ldots,y_n)\mapsto \exp(y_1X_1+\ldots+y_n X_n),
    \end{equation*}
where $\exp \colon \g \to \G$ is the exponential map, which is a global diffeomorphism since $\G$ is nilpotent and simply connected (cf. \cite[Theorem 8.4.7]{Enrico2025book}). We say that $H\subseteq \G$ is a \emph{vertical $k$-codimensional space}, for $k \ge 1$, if there exists a $k$-codimensional subspace $P \subseteq V_1$ such that
    \begin{equation}\label{verticalplanesdefeq}
        H=\exp\left(P\oplus [\g,\g]\right).
    \end{equation}
A vertical $1$-codimensional space is also called \emph{vertical hyperplane}.

We recall that the push-forward $\Phi_*\leb^n$ of the Lebesgue measure $\leb^n$ on $\R^n$ is a Haar measure on $\G$ (cf. \cite[Theorem 8.4.7]{Enrico2025book}). Hereafter, the measure $\leb^n$ also denotes an Haar measure on a stratified group $\G$, where the identification via exponential coordinates of the first type remains understood. If $E \subseteq \G$ is a measurable set, we often use the shorthand notation $\lvert E\rvert$ in place of $\leb^n(E)$.
If $\G$ is a stratified group with associated Lie algebra $\g = V_1\oplus\dots\oplus V_s$, we define the \emph{dilation of factor} $\lambda>0$ as the unique Lie group automorphism $\delta_\lambda \colon \G \to \G$ such that
\begin{equation} \label{dilation_lie_algebra}
\left(\mathrm{d}\delta_\lambda\right)_\id v=\lambda^i v, \quad \text{for all $i\in \{1,\ldots ,s\}$ and $v\in V_i$.}
\end{equation}
We refer to the map $\delta \colon \lambda \mapsto \delta_\lambda$ as the \emph{dilation of $\G$}. If $\G$ and $\hh$ are stratified groups, with dilations $\delta^\G$ and $\delta^\hh$ respectively, we say that a map $f \colon \G \to \hh$ is \emph{homogeneous} if
\begin{equation*}
    f(\delta_\lambda^\G(p))=\delta_\lambda^\hh(f(p)) \quad \text{for all $p\in\G$ and $\lambda >0$.}
\end{equation*} 
Let $\g=V_1 \oplus \dots \oplus V_s$ be a Lie algebra associated with a stratified group $\G$. An inner product $\langle \cdot,\cdot \rangle$ defined on $V_1$ extends by left-translations to a Riemannian metric $g$ defined on the distribution
\begin{equation} \label{def_distribution_carnot}
    \Delta \coloneq \bigcup_{p \in \G} \mathrm{d}L_p(V_1).
\end{equation}
More explicitly, if $v,w \in \Delta_p$, then $g(v,w)\coloneq\langle \mathrm{d}L_p^{-1}(v),\mathrm{d}L_p^{-1}(w)\rangle$. Being $\Delta$ bracket-generating and equiregular, the inner product $\langle \cdot,\cdot \rangle$ endows $\G$ with the structure of equiregular sub-Riemannian manifold. We refer to this distinguished class of sub-Riemannian manifolds as \emph{Carnot groups}.

We stress that, in a Carnot group $\G$ with dilation $\delta$, the Carnot-Carathéodory distance $d_\mathrm{cc}$, defined by \eqref{eq:def_dcc}, satisfies the following:
\begin{itemize}[leftmargin=20pt]
    \item $d_\mathrm{cc}(L_p(q_1),L_p(q_2))=d_\mathrm{cc}(q_1,q_2)$ for all $p,q_1,q_2 \in \G$, \hfill \emph{(left-invariant)}
    \item $d_\mathrm{cc}(\delta_\lambda(q_1),\delta_\lambda(q_2))=\lambda d_\mathrm{cc}(q_1,q_2)$  for all $q_1,q_2 \in \G$ and $\lambda \in \R$. \hfill \emph{(one-homogeneous)}
\end{itemize}
Finally, we denote the open balls centered at $p \in G$ of radius $r>0$ by $B_r(p) \coloneq \set{q\in\G : d_\mathrm{cc}(q,p)<r}$.

\subsection{Embedded surfaces in Carnot groups} \label{embedded_surfaces}

In the following, we consider a \emph{surface} to be a $C^1$-manifold $S \subseteq M$ embedded in a smooth manifold $M$, and a \emph{hypersurface} to be a surface of codimension $1$.
    Let $\G$ be a Carnot group with distribution $\Delta$ defined as in \eqref{def_distribution_carnot}. If $k< \dim(V_1)$, an embedded $k$-codimensional surface $S\subseteq \G$ is \emph{non-characteristic} if $$T_pS + \Delta_p= T_p\G\quad \text{for all $p\in S$.} $$
   In this case, the \emph{horizontal tangent distribution} $\otd$, defined by setting
    \begin{equation*}
        \otdp \coloneq \Delta_p\cap T_p S \qquad \text{for every $p\in S$},
    \end{equation*}
is a distribution of constant rank $\dim(V_1)-k$. We consider on an embedded surface $S \subseteq \G$ the following distance functions.
\begin{definition}
\label{def:distances}
     Let $\G$ be a Carnot group, with Carnot-Carathéodory distance $d_\mathrm{cc}$. Consider an embedded surface $S\subseteq \G$.
     \begin{itemize}[leftmargin=20pt]
         \item The \emph{restricted distance} $d_r:S\times S\to [0,\infty)$ is the restriction of $d_\mathrm{cc} \colon G \times G \to [0,\infty)$ to $S\times S$.
         \item  The \emph{intrinsic distance} $d_i:S\times S\to [0,\infty]$ is defined, for every $p,q\in S$, by setting
     \begin{equation*}
              d_i(p,q):=\inf_\gamma \Big\{\int_0^1\sqrt{g(\dot{\gamma}(t),\dot{\gamma}(t))} \ \df t\Big\},
    \end{equation*}
    where the infimum is taken among all absolutely continuous curves $\gamma \colon [0,1]\to S$, with $\dot\gamma(t)\in\otdgt$ for a.e. $t\in[0,1]$, such that $\gamma(0)=p$ and $\gamma(1)=q$.
     \end{itemize}
\end{definition}

As discussed in \Cref{sectionnonchar}, whilst $d_r$ is always a finite distance on $S$, the intrinsic distance $d_i$, although being positive, symmetric and satisfying the triangle inequality, may not be finite (see \Cref{prop:non-hyper-gen-hyperplane}).

\section{Algebraic structure of hypergenerated groups}
\label{subsec:plentiful}

In this section, we formally introduce the central concept of this work: hypergenerated groups. We then explore their key algebraic properties, provide several examples, and present an effective method to determine whether a given stratified Lie group is hypergenerated.


\begin{definition}[Hypergenerated algebras and groups]
\label{def:plentiful}
A stratified Lie algebra $\g=V_1\oplus\cdots\oplus V_s$ of rank $m$ is \emph{hypergenerated of order $k$} if every linear subspace $P \subseteq V_1$, with $\dim(P) = m-k$, satisfies
\begin{equation*}
    [\mathfrak g,\mathfrak g]\subseteq \Lie (P),
\end{equation*}
where $\Lie(E)$ denotes the smallest Lie sub-algebra of $\g$ containing the subset $E \subseteq \g$. When $k=1$, we simply say that $\g$ is \emph{hypergenerated}. An \emph{hypergenerated group of order $k$} is a stratified group whose Lie algebra is hypergenerated of order $k$.
\end{definition}

By definition, every stratified Lie algebra is hypergenerated of order $0$. 

\begin{remark}
    We stress that the property of being hypergenerated of order $k$ is invariant under isomorphisms of Lie algebras. In particular, it does not depend on the stratification. Indeed, assume that $\mathfrak g_1=V_1\oplus V_2\oplus\cdots\oplus V_s$ and $\mathfrak g_2=W_1\oplus W_2\oplus\cdots\oplus W_s$ are two isomorphic stratified Lie algebras of rank $m$ and step $s$, and that $\g_1$ is hypergenerated of order $k$. Then (see \cite[Proposition 2.17]{primer}), there exists a Lie algebra isomorphism $\varphi \colon \g_1 \to \g_2$ such that $\varphi(V_1)=W_1$. Fix $P \subseteq W_1$, with $\dim(P) =m-k$, then $P'\coloneq\varphi^{-1}(W)$ is a linear subspace of $V_1$ such that $\dim(P')=m-k$. Therefore
    \begin{equation*}
       [\g_2,\g_2]=[\varphi(\g_1),\varphi(\g_1)]=\varphi\left([V_1,V_1]\right) \subseteq \Lie(\varphi(P'))= \varphi(\Lie(P')) =\Lie(P),
    \end{equation*}
    which proves that $\g_2$ is hypergenerated of order $k$ as well.
\end{remark}


\subsection{A characterization of hypergenerated Lie algebras} \label{algebraic_characterization}

\begin{proposition} \label{hyper_step2}
    Let $\g=V_1 \oplus \cdots \oplus V_s$ be a stratified Lie algebra of rank $m$, and $k \ge 0$. The following are equivalent 
    \begin{enumerate}
        \item[(i)] $\g$ is hypergenerated of order $k$;
        \item[(ii)] every subspace $P \subseteq V_1$, with $\dim(P) = m-k$, satisfies $V_2 \subseteq \Lie(P)$. 
    \end{enumerate} 
\end{proposition}

\begin{proof}
    The implication $(i) \Rightarrow (ii)$ is trivial. Assume now that the Lie algebra $\g$ satisfies $(ii)$, and fix $P \subseteq V_1$, with $\dim(P) = m-k$. Since $[\g,\g]=V_2 \oplus \cdots \oplus V_s$, it is sufficient to prove that $V_h \subseteq \Lie(P)$ for every $2 \le h \le s$. We proceed by induction on $h$. The base case $h=2$ is true by assumption. Fix $2 \le h\leq s-1$, and
 assume that $V_n \subseteq \Lie(V)$ for every $2\leq n \le h$. Consider a basis $X_1,\dots,X_{m-k}$ of $P$ and extend it to a basis $X_1,\dots,X_m$ of $V_1$. Since $V_h \subseteq \Lie(P)$ and $h \ge 2$, then 
 $$V_h = \spann \set{ \, [X_i,Y] \, \colon \,   1 \le i \le m-k, \, Y \in V_{h-1} \cap \Lie(P) \,}. $$
 Therefore, since $\g$ is stratified, then
 $$V_{h+1} = \spann \set{ \, [X_j,[X_i,Y]] \, \colon \, 1 \le j \le m, \, 1 \le i \le m-k, \, Y \in V_{h-1} \cap \Lie(P) \,}. $$ 
 We now fix $1 \le j \le m$, $1 \le i \le m-k$, and $Y \in V_{h-1} \cap \Lie(P)$. By the Jacobi identity we get
    \begin{equation}\label{jacobiincharluca}
       [X_j,[X_i,Y]]=[X_i,[X_j,Y]]+[[X_j,X_i],Y]].
    \end{equation}
We recall that $Y\in \Lie(P)$. We observe that $X_i \in P$ and that $[X_j, X_i] \in V_2$, whence $[X_j,X_i] \in \Lie(P)$ by hypothesis. Moreover, $[X_j,Y] \in V_h$ and therefore $[X_j,Y] \in \Lie(P)$ by the inductive hypothesis. We conclude that the right hand side of \eqref{jacobiincharluca} belongs to $\Lie(P)$, and so the left hand side does. Since elements of that form span $V_{h+1}$, then $V_{h+1} \subseteq \Lie(P)$ and the proof follows by induction.
\end{proof}

We recall that the \emph{lower central series} $(\mathfrak g^h)_{h\in \mathbb{N}}$ is the collection of ideals of $\g$ defined inductively by
\begin{equation*}
   \g^1\coloneq\g\quad\text{and}\quad\g^{h+1}\coloneq[\g,\g^h], \quad \text{for any $h\geq 1$.}
\end{equation*}

The following is a consequence of \cref{hyper_step2}, together with the fact that, if $\g=V_1\oplus \dots \oplus V_s$ is a stratified Lie algebra, then $\g/\g^3 = V_1 \oplus V_2$ is a stratified Lie algebra of step $2$.

\begin{corollary} \label{quotient_hypergenerated}
    A stratified Lie algebra $\g$ is hypergenerated of order $k$ if and only if the Lie algebra $\g/\g^3$ is hypergenerated of order $k$.
\end{corollary}

We have now reduced the problem of characterizing hypergenerated Lie algebras to the case of step $2$ stratified Lie algebras. In this regard, it is convenient to introduce the \emph{Kaplan's operator}.
If $V$ is a vector space, we denote by $\sk(V)$ the set of skew-symmetric bi-linear forms on the vector space $V$. Given a stratified Lie algebra $\g = V_1 \oplus V_2 \oplus \dots \oplus V_s$ of step at least $2$, we define the {\em Kaplan's operator} $\J$ (cf. \cite{Kaplan}) to be the linear map $\J \colon V_2^* \to \sk(V_1)$ such that
\begin{equation*} 
	\mu \mapsto \J_\mu, \quad \J_\mu(v,w) \coloneq \mu([v,w]), \quad \text{for every $\mu \in V_2^*$.}
\end{equation*}
Note that, since $V_1$ Lie-generates $\g$, the map $\J$ is injective.
We recall some basic terminology about skew-symmetric bi-linear forms (cf. \cite[Section 6.2]{MR780184}). If $\omega \in \sk(V)$, the {\em rank} of $\omega$ is defined by setting $\rank(\omega) \coloneq \dim(V) - \dim(\ker(\omega))$, 
where 
\begin{equation*}
    \ker (\omega)=\{v\in V\,:\,\omega(v,w)=0\text{ for every }w\in V\}.
\end{equation*}
The rank of a skew-symmetric bi-linear form is always an even number. If $\ker(\omega)$ is trivial, we say that $\omega$ is {\em non-degenerate}. A linear subspace $P \subseteq V$ is said to be {$\omega$-\em isotropic} if $\omega(v,w)=0$ for every $v,w \in P$. 
\begin{remark}
\label{rem:existence_of_isotropic_subspace}
    There exists a $\omega$-isotropic subspace $P \subseteq V$ of codimension $k$ if and only if $\rank(\omega)\le2k$.
\end{remark}

\begin{remark} \label{generate_isotropic}
    We claim that a subspace $P \subseteq V_1$ Lie-generates $V_2$ if and only if, for every $\mu \in V_2^* \setminus \set{0}$, $P$ is not $\J_\mu$-isotropic. Indeed, from the definition of $\J_\mu$, the subspace $P$ is $\J_\mu$-isotropic if and only if $\mu \in [P,P]^\perp \coloneq \set{\eta \in V_2^* : \eta([P,P])=0}$. Moreover, $[P,P]^\perp=\set{0}$ if and only $[P,P]=V_2$, i.e., if and only if $P$ Lie-generates $V_2$.
    \end{remark}

Given a stratified Lie algebra of step $2$, we want to rewrite the property of being hypergenerated in terms of the rank of the bi-linear forms appearing in the image of its Kaplan's operator. Those $2$-step stratified Lie algebras for which $\J_\mu$ is non-degenerate for every $\mu \in V_2^* \setminus \{0\}$ have been widely considered in the literature. They were firstly introduced by Guy Métivier in \cite{metivier} to study the hypo-ellipticity properties of homogeneous operators in stratified groups. We propose a generalized version of his definition.

\begin{definition}
	Let $\g$ be a stratified Lie algebra of step $2$. We define the {\em Métivier order} of $\g$, denoted with $\order(\g)$, to be the number
	\begin{equation} \label{m_order}
		\order(\g) \coloneq \min \set{\rank(\J_\mu) : \mu \in V_2^*\setminus\set{0} }.
	\end{equation}
A stratified Lie algebra $\g$ of step $2$ is {\em Métivier} if $\rank(\g)=\order(\g)$.
\end{definition}

We stress that the Métivier order is always an even number between $2$ and $\rank(\g)$. Our interest in the notion of Métivier order is motivated by the following observation.

\begin{proposition} \label{hyper_order}
    Let $\g$ be a stratified Lie algebra of step $2$, and $k \ge 0$. Then $\g$ is hypergenerated of order $k$ if and only if $2k < \order(\g)$. 
\end{proposition}

\begin{proof}
    Assume that $2k < \order(\g)$ and fix $P \subseteq V_1$ of co-dimension $k$. For every $\mu \in V^*_2 \setminus \set{0}$, since $\rank(\J_\mu) \ge \order(\g) > 2k$, by \cref{rem:existence_of_isotropic_subspace} the subspace $P$ is not isotropic for $\J_\mu$. Therefore, from \cref{generate_isotropic}, the subspace $P$ Lie-generates $V_2$. Thus $\g$ is hypergenerated of order $k$.

    Conversely, assume that $\order(\g) \le 2k$ and fix $\mu \in V_2^*\setminus\set{0}$ such that $\rank(\J_\mu) \le 2k$. Then, by \cref{rem:existence_of_isotropic_subspace}, there exists a subspace $P \subseteq V_1$, of co-dimension $k$, isotropic for $\J_\mu$. By \cref{generate_isotropic}, the subspace $P$ does not Lie-generate $V_2$, therefore $\g$ is not hypergenerated of order $k$.
\end{proof}

\begin{remark}\label{rankprop}
    We claim that, if $\g$ is hypergenerated of order $k$ and step $s\ge2$, then $\rank(\g) \ge 2(k+1)$. Indeed, in view of \cref{hyper_step2}, we can assume that $\g$ is stratified of step $2$. Since $\order(\g)$ is an even number smaller than $\rank(\g)$, then \cref{hyper_order} implies that $2(k+1) \le \order(\g) \le \rank(\g)$.
\end{remark} 

From \cref{hyper_order}, we get the following consequence.
\begin{corollary} \label{res:H-type}
   If $\g$ is a Métivier Lie algebra of rank $2m$, then $\g$ is hypergenerated of order $m-1$.
\end{corollary}

\begin{proposition}\label{firstproperties}
    Let $\g_1$ and $\g_2$ be stratified Lie algebras, and let $\mathfrak{h}$ be an homogeneous ideal of $\g_1=V_1 \oplus \dots \oplus V_s$ (i.e. $[\mathfrak{h},\g_1]\subseteq\mathfrak{h}$ and $\mathfrak{h}=(V_1 \cap \mathfrak{h}) \oplus \dots \oplus (V_s \cap \mathfrak{h})$). The following hold:
    \begin{enumerate}[label = (\roman*)]
        \item The Lie algebra $\g_1\times\g_2$ is hypergenerated of order $k$ if and only if both $\g_1$ and $\g_2$ are hypergenerated of order $k$.
        \item If $\g_1$ is hypergenerated of order $k$, then $\g_1/\mathfrak h$ is hypergenerated of order $k$.
    \end{enumerate}
    \end{proposition}
    \begin{proof} In view of \cref{hyper_step2}, it is not restrictive to assume that $\g_1$ and $\g_2$ are stratified Lie algebras of step at most $2$.
    
    $(i)$ If both $\g_1$ and $\g_2$ are Abelian, the result is straightforward. If $\g_1$ and $\g_2$ are both non-abelian, it follows from \eqref{m_order} that $\order(\g_1 \times \g_2)=\min\set{\order(\g_1),\order(\g_2)}$. Finally, if only $\g_1$ is abelian, then $\order(\g_1\times\g_2)=\order(\g_2)$. In both the latter cases, the result follows from \cref{hyper_order}.
        
    $(ii)$ The stratification of $\g_1/\mathfrak h$ is given by $\g_1/\mathfrak{h}=(V_1+\mathfrak h)/\mathfrak h\oplus (V_2+\mathfrak h)/\mathfrak h$, where $\g_1 =V_1 \oplus V_2$ is the  stratification of $\g_1$. Let $(P + \mathfrak{h})/\mathfrak h$, with $P \subseteq V_1$, be a subspace of $(V_1+\mathfrak h)/\mathfrak{h}$ of codimension $k$. Then, the subspace $P$ has codimension $k$ in $V_1$ and     
    \begin{equation*}
        (V_2+\mathfrak h)/\mathfrak h=([P,P]+\mathfrak h)/\mathfrak h = [(P+\mathfrak h)/\mathfrak h,(P+\mathfrak h)/\mathfrak h ],
    \end{equation*}
    where we used that $\g_1$ is hypergenerated of order $k$.     \end{proof}

We stress that the order of a stratified Lie algebra of step $2$ is precisely the quantity $\widetilde{k}$ appearing in the statement of \cite[Theorem 1.1]{sard_step2} about the abnormal set. By combining the result with \cref{hyper_order}, we get the following:

\begin{corollary}
    Let $\G$ be a stratified group of step $2$, with associated Lie algebra $\g=V_1 \oplus V_2$. Assume that $\G$ is hypergenerated of order $k$, then the abnormal set
    $$\mathrm{Abn}(\G) \coloneq \bigcup \, \set{\exp(\Lie(P)) : P \subseteq V_1, \, [P,V_1] \neq V_2},$$
    is contained in an algebraic variety of codimension at least $2k+3$.
\end{corollary}

\cref{hyper_step2} and \cref{hyper_order} convert the problem of determining whether a stratified Lie algebra of rank $m$ is hypergenerated of a certain order $k$ to the problem of determining if a given linear subspace $S \subseteq \sk(\R^m)$ has a non-trivial element of rank at most $2k$. This kind of problems are known as \emph{MinRank problems} and are deeply studied from the computational complexity point of view, with applications in cryptanalysis (cf. \cites{complexity_minrank,cripto_minrank}). The two problems are actually equivalent in the following sense: on the one hand, given a stratified Lie algebra $\g=\R^m \oplus V_2 \oplus \cdots \oplus V_s$ of rank $m$ and step $s$, we get $S \subseteq \sk(\R^m)$ as the image of the Kaplan's operator. On the other hand, given a linear subspace $S \subseteq \sk(\R^m)$, there exists a stratified Lie algebra of rank $m$ and step $2$ whose image of the Kaplan's operator is $S$. The construction of such Lie algebra is given by the following remark.

\begin{remark} \label{def_gs}
    Given a subspace $S \subseteq \sk(\R^m)$, we define the step $2$ stratified Lie algebra $\g_S \coloneq \R^m \oplus S^*$, where $S^*$ is the dual space of $S$, as follows: for every $v,w \in \R^m$, the Lie bracket $[v,w] \in S^*$ is formally defined by setting
    \begin{equation*}
        [v,w](\omega) \coloneq \omega(v,w) \qquad \text{for every $\omega \in \sk(\R^m)$.}
    \end{equation*}
    With this definition, it is immediate to check that the Kaplan's operator becomes the natural isomorphism $\J \colon S^{**} \to S$ between a finite-dimensional vector space and its double dual.
\end{remark}

The latter discussion suggests that the existence of an easy criterion for the hypergenerating property is unreasonable, since it can be converted to an instance of the MinRank problem, which is known to be NP-hard (see \cite{MR4262585}). The MinRank problem for $S \subseteq \sk(\R^m)$ (identified with $\mathfrak{so}_m$, i.e., the space of $m \times m$ skew-symmetric matrices) has the same complexity as the general case $S \subseteq \mathfrak{gl}_m$. Indeed, every linear subspace in $\mathfrak{gl}_m$ can be mapped in a linear subspace of $\mathfrak{so}_{2m}$ through the map
\begin{equation*}
    f \colon A \mapsto \begin{pmatrix}
        0 & -A \\ A^T & 0
    \end{pmatrix},
\end{equation*}
and the rank of $A$ is less than $k$ if and only if the rank of $f(A)$ is less than $2k$. To conclude, one can implement algorithms that determines wether a given stratified Lie algebra is hypergenerated of order $k$. However, the best known algorithms can only give the answer in exponential time with respect to the size of the input.

\subsection{Hypergenerated groups up to dimension $7$}
We present all hypergenerated groups of topological dimension up to $7$. In view of \cref{firstproperties}, we restrict to {\em indecomposable} Lie groups, i.e., those that are not the direct product of non-trivial stratified groups. We recall that stratified groups of step $1$ are hypergenerated (of order $k$ for every $k \ge 0$). We then focus on indecomposable stratified groups of step $s \ge 2$. As a consequence of \cref{rankprop}, hypergenerated groups of step $s \ge 2$ must have rank at least $2$, thus topological dimension at least $5$. By the classification of stratified Lie groups provided by \cite{MR4490195}, there are only six indecomposable, hypergenerated Lie groups of step $s \ge 2$ and topological dimension between $5$ and $7$. 

A direct computation of the image of their Kaplan's operators shows that the following four stratified groups of step $2$ are Métivier groups. In view of \cref{res:H-type}, they are hypergenerated.
\begin{example} The only indecomposable hypergenerated group of dimension $5$ is the second Heisenberg group. Its associated stratified Lie algebra, of dimension $5$, rank $4$ and step $2$, is $$\g=\spann\{X_1,X_2,X_3,X_4,T\}$$  whose only non-trivial bracket relations are 
\begin{equation*}
    [X_1,X_2]=[X_3,X_4]=T.
\end{equation*}
\end{example}

\begin{example}\label{metiviersei}
    The only indecomposable hypergenerated group of dimension $6$ is the stratified group denoted by $N_{6,4,4a}$. The non-trivial bracket of its associated Lie algebra $$\g=\spann\{X_1,X_2,X_3,X_4,T_1,T_2\}$$ of dimension $6$, rank $4$ and step $2$ are 
        \begin{equation*}
    [X_1,X_3]=[X_2,X_4]=T_1\qquad\text{and}\qquad [X_1,X_4]=[X_3,X_2]=T_2.
\end{equation*}
\end{example}

\begin{example} The third Heisenberg group $\hh^3$ is the only indecomposable stratified group of dimension $7$ that is hypergenerated of order $2$. The only non-trivial bracket relations of its associated Lie algebra $\g=\spann\{X_1,\ldots,X_6,T\}$ of dimension $7$, rank $6$ and step $2$ are
        \begin{equation*}
    [X_1,X_2]=[X_3,X_4]=[X_5,X_6]=T.
\end{equation*}

\end{example}
\begin{example} \label{quaternions}
    We denote by $37D_1$ the stratified group with associated Lie algebra $$\g=\spann\{X_1,\ldots,X_4,T_1,\ldots,T_3\}$$ of dimension $7$, rank $4$ and step $2$ whose only non-trivial bracket relations are 
        \begin{equation*}
    [X_1,X_2]=[X_4,X_3]=T_1,\quad [X_1,X_3]=[X_2,X_4]=T_2\quad\text{and}\quad[X_1,X_4]=[X_3,X_2]=T_3.
\end{equation*}
Both $\hh^2$ and $N_{6,4,4a}$ can be seen as quotients of $37D_1$.
\end{example}

The latter two examples 
are not Métivier groups.
\begin{example}\label{example}
We consider the stratified group $27B$ associated to the Lie algebra $$\g=\spann\{X_1,\ldots,X_5,T_1,T_2\}$$ of dimension $7$, rank $5$ and step $2$ whose only non-trivial braket relations are
\begin{equation*}    
[X_1,X_2]=[X_3,X_4]=T_1,\quad[X_1,X_5]=[X_2,X_3]=T_2.
\end{equation*}
Since its rank is odd, the stratified group $27B$ is not a Métivier group. Nevertheless, as proved in \cite{psv}, it is hypergenerated. 
%
%
\end{example}

\begin{example}
    The only indecomposable hypergenerated group of step $3$ and dimension $7$ is the group denoted by $137A_1$. The only non-trivial bracket relations of its associated Lie algebra $\g=\spann\{X_1,\ldots,X_4,T_1,T_2,S\}$ of dimension $7$, rank $4$ and step $3$ are 
        \begin{equation*}
    [X_1,X_3]=[X_2,X_4]=T_1,\quad [X_1,X_4]=[X_3,X_2]=T_2\quad\text{and}\quad [X_1,T_1]=[X_2,T_2]=S.
\end{equation*}
It is easy to see that $\g/\spann\{S\}$ is isomorphic to the Lie algebra associated to $N_{6,4,4a}$, which is hypergenerated. Thus, by  \cref{quotient_hypergenerated}, the Lie algebra $137A_1$ is hypergenerated. 
\end{example}

\subsection{Existence of hypergenerated Lie algebras of large step}

We conclude this section by discussing how the property of being hypergenerated of order $k$ does not put a constrain on the step of the Lie algebra. We first focus on the case $k=1$.

\begin{proposition} \label{existencehypergenerated}
    For every $m \ge 4$ and $s \ge 1$ there exists a hypergenerated Lie algebra of rank $m$ and step $s$.
\end{proposition}

\cref{existencehypergenerated} follows from the following remarkable example.

\begin{example} \label{quaternionic}
  Consider the free-nilpotent Lie algebra $\f$ of rank $3$ and step $s$ (see \cite[Example 2.5]{primer}), together with a stratification $\f = V_1 \oplus \cdots \oplus V_s$. Fix a basis $\set{X_1,X_2,X_3}$ of $V_1$ and define a linear map $\phi \colon V_1 \to V_2$ by setting $\phi(X_1)=[X_2,X_3]$, $\phi(X_2)=[X_3,X_1]$, and $\phi(X_3)=[X_1,X_2]$. By \cite[Proposition 3]{free_extension}, the map $\phi$ extends to a derivation $\widetilde{\phi} \colon \f \to \f$ (i.e. $\widetilde{\phi}([Y_1,Y_2])=[\widetilde{\phi}(Y_1),Y_2]+[Y_1,\widetilde{\phi}(Y_2)]$ for every $Y_1,Y_2 \in \f$). We then consider the semi-direct sum by $\widetilde{\phi}$, i.e., the Lie algebra $\g \coloneq \f \oplus \R$ with Lie brackets given, for every 
  $Y_1,Y_2 \in \f$ and $t_1,t_2 \in \R$, by
  {\begin{equation*}
      [(Y_1,t_1),(Y_2,t_2)]_\g=([Y_1,Y_2]_\f + t_1\widetilde{\phi}(Y_2) - t_2\widetilde{\phi}(Y_1),0).
      \end{equation*}}

Since $\widetilde{\phi}(V_1) \subseteq V_2$, we obtain that $\g$ is again a stratified Lie algebra with stratification $\g = (V_1 \oplus \R) \oplus V_2 \oplus \cdots \oplus V_s$. Therefore $\g$ is a stratified Lie algebra of rank $4$ and step $s$. Moreover, it is immediate to check that $\g/\g^3$ is isomorphic to the Lie algebra in \cref{quaternions}, which is hypergenerated. By \cref{quotient_hypergenerated}, we conclude that $\g$ is hypergenerated as well.
\end{example}

\begin{remark}
    We stress another property of the Lie algebra $\g$ defined in \cref{quaternionic}. For every $1$-codimensional subspace $P \subseteq (V_1 \oplus \R)$, the stratified Lie algebra $\Lie(P)=P \oplus [\g,\g]$ has the same rank, step, and dimension of the free-nilpotent Lie algebra $\f$ of rank $3$ and step $s$, therefore $\Lie(P)$ is isomorphic to $\f$. We conclude that all stratified Lie algebras in the family
    \begin{equation*}
        \set{ \Lie(P) \, : \, P \subseteq (V_1 \oplus \R), \, \dim(P)=3 }
    \end{equation*}
    are isomorphic.
\end{remark}

\begin{proof}[Proof of \cref{existencehypergenerated}]
  Fix $m \ge 4$ and $s \ge 1$. Let $\g$ be the Lie algebra of rank $4$ and step $s$ given by \cref{quaternionic}. Then the Lie algebra $\g \times \R^{m-4}$ has rank $m$ and step $s$. Moreover, by \cref{firstproperties}, the Lie algebra $\g$ is hypergenerated.
\end{proof}

Next, we discuss the case of hypergenerated Lie algebras of order higher than $1$. We prove the following fact.

\begin{proposition} \label{existence_k_hypergenerated}
    For every $k \ge 2$ and $s \ge 2$, there exists an hypergenerated Lie algebra of step $s$, and order $k$.
\end{proposition}

We obtain the Lie algebra in the statement of \cref{existence_k_hypergenerated} as some quotient of a \emph{free-metabelian Lie algebra}, which we now introduce.

\begin{definition} \label{free_metabelian}
    Fix $s,m \in \N$. The \emph{free-metabelian Lie algebra of step $s$, generated by $X_1,\dots,X_m$}, is the stratified Lie algebra
    $\m = V_1 \oplus \dots \oplus V_s$, where $X_1,\dots,X_m$ is a basis of $V_1$ and, for every $2 \le k \le s$, a basis of $V_k$ is given by the formal elements
    \begin{equation} \label{basis_free_metabelian}
        (X_{i_1},X_{i_2},\dots,X_{i_k}) \quad \text{with $1 \le i_h \le m, \, i_1 > i_2 \le \cdots \le i_k$}
    \end{equation}
    and the only non-trivial brackets relations between elements of the basis are given by
    \begin{enumerate}
        \item $[X_i,X_j]=(X_i,X_j)$ if $i>j$, 
        \item $[(X_{i_1},X_{i_2},\dots,X_{i_k}),X_j]=(X_{i_1},X_{i_2},\dots,X_{i_k},X_j)$ if $2 \le k \le s-1$ and $i_k < j$,
        \item $[(X_{i_1},X_{i_2},\dots,X_{i_k}),X_j]=(X_{i_1},X_{i_2},\dots,X_{i_h},X_j,X_{i_{h+1}},\dots,X_{i_k})$ if $2 \le k \le s-1$ and $i_h \le j \le i_{h+1}$ for some $2 \le h \le k$,
        \item $[(X_{i_1},X_{i_2},\dots,X_{i_k}),X_j]=(X_{i_1},X_j,X_{i_2},\dots,X_{i_k})-[(X_{i_2},X_j,X_{i_3},\dots,X_{i_k}),X_{i_1}]$ if $2 \le k \le s-1$ and $i_2 > j$.
    \end{enumerate}
We also refer to such Lie algebra $\m$ as the \emph{free-metabelian Lie algebra of step $s$ and rank $m$}. 
\end{definition}

\begin{remark}
  In the latter definition, $\m$ is indeed a Lie algebra (see \cites{MR159847,MR2905023}), in particular the Jacobi identity is satisfied. Moreover, the Lie algebra $\m$ is \emph{metabelian} (or \emph{$2$-step solvable}), i.e., $[\m^{2},\m^{2}]=\set{0}$. We claim that 
  \begin{equation} \label{adjoint_commute}
  [[Y,Z_1],Z_2]=[[Y,Z_2],Z_1] \quad \text{for every $Y \in \m^2$ and $Z_1,Z_2 \in \m$.}    
  \end{equation}
  Indeed, by the Jacobi identity, $[[Y,Z_1],Z_2]-[[Y,Z_2],Z_1]=[[Z_2,Z_1],Y] \in [\m^{2},\m^{2}]=\set{0}$.
\end{remark}



We consider free-metabelian Lie algebras because it is easy to compute the dimension of each layer of its stratification, as well as estimating the growth of ideals generated by subspaces of the second layer. This is clarified by the following two lemmas.

\begin{lemma}
    Let $\m = V_1 \oplus \cdots \oplus V_s$ be the free-metabelian Lie algebra of step $s$ and rank $m$, we define $d_k^m \coloneq \dim(V_k)$. Then
    \begin{equation} \label{estimate_dkm}
        d_k^m=(k-1)  \binom{m+k-2}{k} \quad \text{for every $2 \le k \le s$.}
    \end{equation}
\end{lemma}  

\begin{proof}
It follows by computing the number of elements in $\eqref{basis_free_metabelian}$. More precisely, for every $1 \le h \le m-1$, the elements in $\eqref{basis_free_metabelian}$ with $i_2=m-h$ are
\begin{equation*}
    h  \binom{h+k-2}{h}.
\end{equation*}
Therefore
\begin{align*}
    d_k^m &= \sum_{h=1}^{m-1} h  \binom{h+k-2}{h} \\
    &= \sum_{h=1}^{m-1} h  \binom{h+k-1}{h} - \sum_{h=1}^{m-1} h  \binom{h+k-2}{h-1} \\
    &=(m-1)\binom{m+k-2}{m-1}-\sum_{h=1}^{m-1}\binom{h+k-2}{h-1}  \\ &= (m-1)\binom{m+k-2}{m-1}-\sum_{h=1}^{m-1}\binom{h+k-1}{h-1}+ \sum_{h=1}^{m-1} \binom{h+k-2}{h-2} \\
    &= (m-1)\binom{m+k-2}{m-1} - \binom{m+k-2}{m-2} \\
    &= (k-1)  \binom{m+k-2}{k} \, . \qedhere
\end{align*}
\end{proof}

\begin{remark} \label{kaplan_iso}
    From \eqref{estimate_dkm}, we get that $\dim(V_2)=\binom{m}{2}=\dim(\sk(V_1))$, therefore the Kaplan's operator $\J \colon V_2^* \to \sk(V_1)$ defines an isomorphism. We infer that, for every subspace $S \le \sk(V_1)$, there exists a unique subspace $W\subseteq V_2$, with $\dim(W)=\dim(\sk(V_1)) - \dim(S)$, such that $\J(W^\perp)=S$, where $W^\perp \coloneq \set{\varphi \in V^*_2: \varphi(W)=0}$ is the annihilator of $W$.
\end{remark}

\begin{lemma}
    Let $\m = V_1 \oplus \cdots \oplus V_s$ be the free-metabelian Lie algebra of step $s$, generated by $X_1,\dots,X_m$, and $W$ be a subspace of $V_2$. We denote by $I(W)$ be the ideal of $\m$ generated by $W$ and we define $c_k^m(W) \coloneq \dim(I(W) \cap V_k)$. Then 
    \begin{equation} \label{estimate_ckm}
        c_k^m(W) \le \dim(W)  \binom{m+k-3}{k-2}.
    \end{equation}
\end{lemma}   

\begin{proof}
    Let $W_1,\dots,W_{\dim(W)}$ be a basis of $W$. Since $W \subseteq \m^2$, by \eqref{adjoint_commute} and the fact that $\m$ is stratified and generated by $X_1,\dots,X_m$, we get that elements of the form
    \begin{equation*} 
        [[\cdots[W_j,X_{i_1}],\cdots],X_{i_{k-2}}] \quad \text{with $1 \le j \le \dim(W)$ and $i_1 \le \dots \le i_{k-2}$}
    \end{equation*}
    generate $I(W) \cap V_k$. The result follows by computing the number of those elements.
\end{proof}



We are finally ready for the final proof of this section:

\begin{proof}[Proof of \cref{existence_k_hypergenerated}]
    Fix $k \ge 2$, $s \ge 2$, and $m$ to be chosen large enough, so that
    \begin{equation} \label{condition m}
    2km - 2k^2 - k < \frac{(m+s-2)(m-1)}{s},
\end{equation}
is satisfied. Consider the free-metabelian Lie algebra $\m = V_1 \oplus \cdots \oplus V_s$ of step $s$ and rank $m$. Following the proof of \cite[Theorem 1]{large_rank} in the skew-symmetric setting, we can select a linear subspace $S \subseteq \sk(V_1)$, with $\dim(S)=(m-2k)(m-2k-1)/2$, such that all non-zero elements of $S$ have rank at least $2k+2$. By \cref{kaplan_iso}, there exists a linear subspace $W \subseteq V_2$ such that $\J(W^\perp)=S$, where $\J$ is the Kaplan's operator of $\m$, and 
    \begin{equation} \label{eq_dimW}
        \dim(W) = \dim(\sk(V_1)) - \dim(S) = \binom{m}{2} - \binom{m-2k}{2}= 2km - 2k^2 - k .
    \end{equation}
    
Denote by $I \coloneq I(W)$ the ideal in $\m$ generated by $W$. Since $W$ is dilation invariant (i.e. invariant under the maps defined in \eqref{dilation_lie_algebra}), we get that $I$ is dilation invariant as well. Therefore, dilation maps commute with the quotient map $\pi \colon \m \to \m/I$ and the structure of stratified Lie algebra is preserved. The stratification is given by:
\begin{equation*}
    \m/I = V_1 \oplus V_2/W \oplus V_3/(I \cap V_3) \oplus \cdots \oplus V_s/(I \cap V_s) \,.
\end{equation*}

We claim that $\m/I$ is hypergenerated of order $k$. Indeed, the image of the Kaplan's operator $\widetilde{\J}$ of $\m/I$ is
\begin{equation*}
    \widetilde{\J} \left( \left(V_2/W\right)^*\right) = \J(W^\perp)=S,
\end{equation*}
where we used the natural identification $\left(V_2/W\right)^* \cong W^\perp \subseteq V_2^*$ and the fact that $\m$ and $\m/I$ have the same first stratum. Since all non-zero elements of $S$ have rank at least $2k+2$, then $\order(\m/I) \ge 2k+2$. By \cref{hyper_order} we get that $\m/I$ is hypergenerated of order $k$. 

It is left to prove that $\m/I$ is stratified of step $s$. This is the case if and only if $\dim\left(V_s/(I \cap V_s)\right) > 0$, i.e., $\dim(V_s) > \dim(I \cap V_s)$. By combining \eqref{estimate_dkm} and \eqref{estimate_ckm}, it is sufficient to verify that
\begin{equation*}
    \dim(W)  \binom{m+s-3}{s-2} < (s-1)  \binom{m+s-2}{s}.
\end{equation*}
In view of \eqref{eq_dimW}, and after some easy computations, the latter inequality is equivalent to \eqref{condition m}. We conclude that $\m/I$ is stratified of step $s$. The statement is then proved by setting $\g \coloneq \m/I$. \end{proof}

\section{Hypersurfaces with locally constant normal}
\label{sec:locally_const_normal}
In this section we show that the class of hypergenerated groups constitutes the correct environment to deal with regularity issues for perimeter minimizers in stratified Groups. We briefly recall some basic preliminaries, for which we refer to \cite{MR3587666}. In the following, we fix a basis $X_1,\ldots,X_m$ of 
 the first layer $V_1$ of the stratified Lie algebra $\mathfrak{g}=V_1\oplus\ldots \oplus V_s$, and we endow $\g$ with a left-invariant Riemannian metric $\langle\cdot,\cdot\rangle$ for which $X_1,\ldots,X_m$ is an orthonormal frame. Moreover, we identify $\G$ with $\rr^n$ via exponential coordinates of the first type.
 We fix an open set $\Om\subseteq\G$. A measurable set $E\subseteq\G$ is  \emph{of locally finite $\G$-perimeter} in $\Om$, or equivalently a \emph{$\G$-Caccioppoli set} in $\Om$, if
    \begin{equation}\label{caccioppolig}
    \sup\left\{\int_E \divv_\G(\varphi)\,dx\,:\,\varphi\in \Gamma_c(\tilde\Om,\Delta),\, \langle\varphi,\varphi\rangle\leq 1\right\}<+\infty,
\end{equation}
for all open set $\tilde\Om\Subset\Om$. Here and hereafter, the set $\Gamma_c(\tilde\Om,\Delta)$ denotes the family of smooth horizontal vector fields which are compactly supported in $\Om$, while $\divv_\G$ is the \emph{horizontal divergence}, defined by setting
\begin{eqnarray*}
    \divv_\G\left(\sum_{j=1}^m\varphi_jX_j\right):=\sum_{j=1}^mX_j\varphi_j.
\end{eqnarray*}
In addition, $E$ is of \emph{finite $\G$-perimeter} in $\Om$ if \eqref{caccioppolig} holds with $\tilde\Om=\Om$.
As in the Euclidean setting, if $E$ is a $\G$-Caccioppoli set in $\Om$, Riesz theorem implies the existence of a radon measure $P_\G(\cdot,\Om)$ on $\Om$ and of a $P_\G(\cdot,\Om)$-a.e. unique measurable horizontal vector field $\vg$ satisfying
\begin{equation*}
   \left\langle \vg,\vg\right\rangle=1,
\end{equation*} 
$P_\G(\cdot,\Om)$-a.e. in $\in\Om$ and
\begin{equation*}
   \int_E\divv_\G\varphi\,dx=-\int_\Om\langle\vg,\varphi\rangle\,dP_\G(\cdot,\Om),
\end{equation*}
for every $\varphi\in \Gamma_c(\Om,\Delta)$. The measure $P_\G(\cdot,\Om)$ is called the \emph{$\G$-perimeter measure} of $E$ in $\Om$, while $\vg$ is the (measure theoretic inner) \emph{horizontal unit normal} to $E$ in $\Om$. The \emph{$\G$-reduced boundary} $\partial^*_\G E$ of $E$ is the set of points $p\in\Om$ such that
    \begin{equation*}
        P_\G(E,B_r(p))>0,
    \end{equation*}
    for every $r>0$, the limit
    \begin{equation*}
       \lim_{r\to 0^+}\ave_{B_d(p,r)}\nu_\G\,dP_\G(E,\cdot)
    \end{equation*}
    exists, and its $\langle\cdot,\cdot\rangle$-norm is equal to $1$.
The \emph{measure-theoretic} boundary of a measurable set $E\subseteq\G$ is 
\begin{equation}
\label{eq:boundary}
\partial E
=
\set*{p\in \G : |E\cap B_r(p)|>0\
\text{and}\
|E^c\cap B_r(p)|>0\
\text{for all}\ r>0
}.
\end{equation}
Up to modifying a set $E\subseteq\G$ of locally finite $\G$-perimeter in an $\leb^n$-negligible way, arguing \emph{verbatim} as in \cite[Proposition 12.19]{MR2976521}, we can always assume that $\partial E$ coincides with the topological boundary of $E$, and that
    \begin{equation}\label{interiordensity}
        E=\left\{p\in\mathbb \G\,:\,\liminf_{r\to 0}\frac{|E\cap B_r(p)|}{|B_r(p)|}>0\right\}.
    \end{equation}
Finally, if $q\in\partial E$ and $r>0$, we say that $E$ has \emph{locally constant normal} at $q$ inside $B_r(q)$ if there exists a unit left-invariant vector field $\nu$ for which
   	$\vg(x)=\nu(x)$ for $|\partial E|_{\mathbb{G}}$-a.e. $x\in B_r(q)$.  With the next result, we extend \cite[Theorem 3.6]{psv} to hypergenerated groups of arbitrary step, proving the equivalence $(1)\Leftrightarrow(2)$ of \Cref{thm:main-theorem}. 

\begin{theorem}\label{hgiffchnimplhp}
    Let $\G$ be a Carnot group. The following are equivalent.
    \begin{itemize}
        \item [(i)] $\G$ is hypergenerated.
        \item [(ii)] For every set $E$ of finite $\mathbb{G}$-perimeter in $B_r(q)$, with $q\in\partial E$ and $r>0$, such that $E$ has locally constant normal at $q$ inside $B_r(q)$,
   	the boundary $\partial E$ is the vertical hyperplane orthogonal to $\nu$ inside $B_r(q)$.
    \end{itemize}
\end{theorem}

To prove the sufficiency of the hypergenerated property, we begin with a preliminary lemma inspired by \cite[Lemma 2.1]{MR2875838}.
\begin{lemma}\label{montivittone}
    Let $\G$ be a Carnot group. Let $E$ be a set of finite $\mathbb G$-perimeter in $ B_r(0) $ for some $r>0$. Let $Z$ be a horizontal left-invariant vector field such that 
    \begin{equation*}
        \int_E Z\psi(p)\,dp\leq 0,
    \end{equation*}
    for every $\psi\in C^1_c( B_r(0) )$ such that $\psi\geq 0$. Then
    \begin{equation*}
        \exp (sZ)(A)\subseteq B_r(0) \quad\Rightarrow\quad \mathcal L^{n}(E\cap A)\leq\mathcal L ^{n}(E\cap \exp(sZ)(A))
    \end{equation*}
    for every $s>0$ and for every $\mathcal L^{n}$-measurable set $A\subseteq  B_r(0) $.
\end{lemma}
\begin{proof}
    It follows exactly as in the proof of \cite[Lemma 2.1]{MR2875838}; see also \cite{LeDonne-Bellettini2}.
\end{proof}
   \begin{proposition}\label{analogadimonti3.6}
   	Let $\G$ be a hypergenerated group. Let $E$ be a set of finite $\mathbb{G}$-perimeter in $B_r(q)$, with $q\in\partial E$ and $r>0$. If $E$ has locally constant normal at $q$ inside $B_r(q)$,
   	  then $\partial E$ is the vertical hyperplane orthogonal to $\nu$ inside $B_r(p)$.
   	\end{proposition}  
\begin{proof}
    We follow the proof of \cite[Proposition 3.6]{MR3194680} (see also \cite{LeDonneKleinerAmbrosio}). Up to left-translations, we can assume that $q=0$.
    We set $\tilde P=\spann\{\nu\}^\perp\cap V_1$. Then, since $\mathbb G$ is hypergenerated, we know that $[\g,\g]\subseteq Lie(\tilde P)$. Let $\tilde X_1,\ldots,\tilde X_{m_1}$ be a basis of $\tilde V$ that generates $[\g,\g]$ by commutation. We identify the latter with their left-invariant extensions.  Fix $\psi\in C^1_c( B_r(0) )$, set $\phi:=\psi\nu$ and for each $j=1,\ldots,m-1$, set $\phi^{\pm}_j=\pm\psi \tilde X_j$. Then it holds that
    \begin{equation*}
    \begin{split}
       \int_E\nu\psi\,d\mathcal L^n&=   \int_E\diver_{\mathbb G}\phi\,d\mathcal L^n\\
       &=-\int_{\mathbb G}\langle\phi,\vg\rangle\,dP_\G(E,\cdot)\,\\
       &=-\int_{\mathbb G}\langle\phi,\nu\rangle\,dP_\G(E,\cdot)\,\\
       &=-\int_{\mathbb G}\psi\,dP_\G(E,\cdot)\\
       &\leq 0,
    \end{split}
    \end{equation*}
    and
        \begin{equation*}
        \begin{split}
             \pm \int_E\tilde X_j\psi\,d\mathcal L^n= \int_E\diver_{\mathbb G}\phi^\pm_j\,d\leb ^n
             =-\int_{\mathbb G}\langle\phi^\pm_j\vg\rangle\,dP_\G(E,\cdot)\,
             =-\int_{\mathbb G}\langle\phi^\pm_j,\nu\rangle\,dP_\G(E,\cdot)
             =0
        \end{split}    
    \end{equation*}
    for every $j=1,\dots,m-1$. Therefore, arguing as in the proof of \cite[Proposition 3.6]{MR3194680} and thanks to \Cref{montivittone}, we conclude that 
    \begin{equation*}
        p\in E\cap B_r(0) \text{ and }\exp (\pm s \tilde X_j)(p)\in B_r(0) \quad\Rightarrow\quad\exp (\pm s\tilde X_j)(p)\in E
    \end{equation*}
    and
    \begin{equation*}
        p\in E\cap B_r(0) \text{ and }\exp (s \nu)(p)\in B_r(0) \quad\Rightarrow\quad\exp (s\nu)(p)\in E
    \end{equation*}
    for every $s>0$ and every $j=1,\ldots,m-1$. The thesis follows as in the proof of \cite[Proposition 3.6]{MR3194680}.
\end{proof}
When, instead, the Carnot group $\G$ is not hypergenerated, it is possible to exhibit smooth counterexamples. Since our construction applies in the same way to provide counterexamples in $k$-hypergenerated groups of arbitrary order, we give here a unified statement. If $S\subseteq \G$ is a $k$-codimensional embedded surface and $k>1$, we can no longer speak about constant horizontal normal. Nevertheless, the latter property can be rephrased looking at the horizontal tangent space. Indeed, if $S\subseteq\G$ is a smooth, embedded, non-characteristic hypersurface and if we denote by $\v$ its horizontal unit normal, then $\v$ is constant if and only if there exists a $1$ codimensional subspace $P\subseteq V_1$ such that
\begin{equation}\label{costanttangentforhypers}
    \otc{p}=\mathrm{d} L_p P,\qquad \text{for every $p\in S$}.
\end{equation}
Let now $S\subseteq \G$ be a smooth, embedded, non-characteristic surface of codimension $1\leq k\leq m-2$. Notice that, since $S$ is non-characteristic, then
\begin{equation*}
    \dim (\otc{p})=m-k,\qquad \text{for every $p\in S$.}
\end{equation*}
 According to \eqref{costanttangentforhypers}, we say that $S$ has \emph{locally constant horizontal tangent space} near $p\in S$ if there exists $r>0$ and a $k$-codimensional subspace $P\subseteq V_1$ such that
\begin{equation}\label{costanttangentforsurf}
    \otc{q}=\mathrm{d} L_qP,\qquad \text{for every $q\in B_r(p)$.}
\end{equation}
 
\begin{proposition}\label{counterexchnnothpcodimalta}
    Let $\G$ be a Carnot group. Let $1\leq k\leq m-2$. Assume that $\G$ is not $k$-hypergenerated. Then there exists a smooth, embedded, non-characteristic $k$-codimensional hypersurface $S$, $p\in S$, and $r>0$ such that $S$ has locally constant horizontal tangent space in $B_r(p)$ but $S\cap B_\varrho(p)$ is not a vertical $k$-codimensional plane for every $0<\varrho\leq r$. 
\end{proposition}
\begin{proof}
    Let $n$ be the dimension of the Lie algebra $\mathfrak g$ of $\mathbb G$. Assume that $\G$ is not $k$-hypergenerated. Then there exists a $k$-codimensional subspace $P\subseteq V_1$ such that $[\g,\g]\not\subseteq\Lie(P)$. We fix a basis $$(X_1,\ldots,X_k,Y_1\ldots,Y_{m-k})$$ of $V_1$ in such a way that 
    \begin{equation*}
        P=\spann\{Y_1,\ldots,Y_{m-k}\}\qquad\text{and}\qquad V_1=\spann\{X_1,\ldots,X_k\}\oplus P.
    \end{equation*} Since $[\g,\g]\not\subseteq\Lie(P)$, there exists a non-empty subspace $W\subseteq [\g,\g]$ such that $$\g=\spann\{X_1,\ldots,X_k\}\oplus \Lie (P)\oplus W.$$ We extend $(X_1,\ldots,X_k,Y_1\ldots,Y_{m-k})$ to a basis $(X_1,\ldots,X_k,,Y_1\ldots,Y_{\alpha},W_1,\ldots,W_\beta)$ of $\mathfrak g$, for suitable $\alpha,\beta\in\mathbb N\setminus\{0\},$ in such a way that $\Lie(P)=\spann\{Y_1,\ldots,Y_\alpha\}$ and $W=\spann\{W_1,\ldots,W_\beta\}$. We identify elements of $\mathfrak g$ with their left-invariant extension. 
    We identify $\G$ with $\mathbb R^n$ via exponential coordinates. More precisely, if $$v=\sum_{j=1}^kx_j X_j+\sum_{j=1}^\alpha y_j Y_j+\sum_{j=1}^\beta w_j W_j\in\g,$$ then we identify $v$ with the point $p=(x_1,\ldots,x_k,y_1,\ldots,y_\alpha,w_1,\ldots,w_\beta)\in\rr^n$. Let $$\varphi_1,\ldots,\varphi_k:\rr^\beta\to \rr$$ be any smooth functions such that, for every $j=1,\ldots,k$,  $\varphi_j(\bar 0)=0$ and 
  $D\varphi_j(\bar w)=\bar 0$ if and only if $\bar w=0$ (for instance, $\varphi_j(\bar w)=|\bar w|^2$ for every $j=1,\ldots, k$). Consider the set 
    \begin{equation*}
        E_\varphi=\left\{(\varphi_1(\bar w),\ldots,\varphi_k(\bar w),\bar 0,\bar 0)\cdot(0,\bar 0,\bar w)\cdot (0,\bar y,\bar 0)\,:\,\bar y\in\rr^\alpha,\,\bar w\in\rr^\beta\right\}.
    \end{equation*}
    Notice that $E_\varphi$ is a smooth, embedded $k$-codimensional surface such that $0\in S$. Moreover, by our choice of $\varphi$, $E_\varphi$ is not flat in every neighborhood of $0$. We claim that $Y_j|_p\in T_p S$ for every $i=1,\ldots, \alpha$ and every $p\in S$. Indeed, fix $p\in S$ and define the curve $\gamma (t)=p\cdot\exp (t Y_i)$ for $t\in\rr$. Notice that $\gamma(0)=p$ and $\Dot\gamma(0)=Y_i|_p$. Moreover, notice that
    \begin{equation*}
    \begin{split}
          \gamma(t)&=p\cdot\exp(tY_i)\\
          &=(\varphi_1(\bar w),\ldots,\varphi_k(\bar w),\bar 0,\bar 0)\cdot(0,\bar 0,\bar w)\cdot (0,\bar y,\bar 0)\cdot\exp(tY_i)\\
          &=(\varphi_1(\bar w),\ldots,\varphi_k(\bar w),\bar 0,\bar 0)\cdot(0,\bar 0,\bar w)\cdot\exp\left(\sum_{j=1}^\alpha y_j Y_j\right)\cdot\exp(tY_i)\\
          &=(\varphi_1(\bar w),\ldots,\varphi_k(\bar w),\bar 0,\bar 0)\cdot(0,\bar 0,\bar w)\cdot\exp\left(\left(\sum_{j=1}^\alpha y_j Y_j\right)\star \left(tY_i\right)\right),
    \end{split}
    \end{equation*}
    where we denoted by $\star$ the Dynkin on $\g$ using the Campbell-Hausdorff formula (cf. \cite[Section 4.7.3]{Enrico2025book}). In particular, since $\star\left(\Lie(P)\times\Lie(P)\right)\subseteq \Lie (P)$, we conclude that $\gamma(t)\in E_\varphi$, so that $Y_i|_p\in T_p S$. We claim that there exists $r>0$ such that $S\cap  B_r(0) $ is non-characteristic. Being $S$ smooth, it suffices to show that $0$ is a non-characteristic point, or equivalently that $X_1|_0,\ldots, X_k|_0\notin T_0 S$. Let $\phi:\rr^\beta\times\rr^\alpha\longrightarrow\rr^n$ be defined by
    \begin{equation*}
        \phi(\bar w,\bar y)=(\varphi_1(\bar w),\ldots,\varphi_k(\bar w),\bar 0,\bar 0)\cdot(0,\bar 0,\bar w)\cdot (0,\bar y,\bar 0),\qquad\text{for  $\bar w\in\rr^\beta$ and $\bar y\in\rr^\alpha$.}
    \end{equation*}
    Then a simple computation shows that
    \begin{equation*}
        \frac{\partial\phi}{\partial w_i}\Big |_{(\bar 0,\bar 0)}=W_i|_0\qquad\text{and}\qquad\frac{\partial\phi}{\partial y_i}\Big |_{(\bar 0,\bar 0)}=Y_j|_0,
    \end{equation*}
    for every $i=1,\ldots,\beta$ and every $j=1,\ldots,\alpha$. In particular, $X_1|_0,\ldots, X_k|_0\notin T_0 S$, so that there exists $r>0$ such that $S\cap  B_r(0) $ is non-characteristic. Since we already know that $Y_1|_p,\ldots,Y_{m-k}|_p\in T_p S$ for every $p\in S$, we conclude that
    \begin{equation*}
        \otc{p}=\spann\{Y_1|_p,\ldots, Y_{m-k}|_p\}
    \end{equation*}
    for every $p\in S\cap  B_r(0) $, so that $S$ has locally constant horizontal tangent space in $B_r(0)$. Nevertheless, our choice of $\varphi$ implies that $S\cap B_\varrho(0)$ is not a vertical $k$-codimensional plane for every $0<\varrho\leq r$.
\end{proof}
\begin{proof}[Proof of \Cref{hgiffchnimplhp}]
    The theorem follows combining \Cref{analogadimonti3.6} and \Cref{counterexchnnothpcodimalta}.
\end{proof}

In the smooth setting, \Cref{analogadimonti3.6} generalizes to higher codimension as follows.
\begin{proposition}\label{sufficiencysec3highcodim}
   Let $1\leq k\leq m-2$. Let $\G$ be a $k$-hypergenerated group. Let $S$ be a smooth, embedded, non-characteristic surface of codimension $k$. Let $p\in S$ and $r>0$ be such that $S$ has locally constant horizontal tangent space in $B_r(p)$. Then $S\cap B_r(p)$ is a vertical $k$-codimensional plane. 
\end{proposition}
\begin{proof}
    Let $S,p$ and $r$ be as in the statement. By hypothesis, there exists a $k$-codimensional subspace $P\subseteq V_1$ such that \eqref{costanttangentforsurf} holds for every $q\in B_r(p)$. Let us define
    \begin{equation*}
        H:=\exp\left(P\oplus [\g,\g]\right).
    \end{equation*}
    According to \eqref{verticalplanesdefeq},the space $H$ is a vertical $k$-codimensional plane. Since $\G$ is $k$-hypergenerated, then
    \begin{equation}\label{pianiinhypergen}
        H=\exp (\Lie(P)).
    \end{equation}
    Fix $q\in S\cap B_r(p)$. To conclude the proof, it suffices to show that there exists $\varrho>0$ such that
    \begin{equation*}
        q\cdot H\cap B_\varrho (q)\subseteq S.
    \end{equation*}
    Up to left-translating $S$, we may assume that $q=0$. In view of \eqref{pianiinhypergen}, for a fixed $\tilde q\in H$ there exists $Z\in\Lie (P)$ such that
    \begin{equation}\label{tildequgualeexp}
        \tilde q=\exp (Z).
    \end{equation}
Let us identify $Z$ with its left-invariant extension. Since 
$d\mathrm L_{\tilde p}P\subseteq T_{\tilde p}S$ 
for every $\tilde p\in S\cap B_r(p)$, then $d\mathrm L_{\tilde p}\lie (P)\subseteq T_{\tilde p}S$
for every $\tilde p\in S\cap B_r(p)$, whence
$Z_{\tilde p}\in T_{\tilde p}S$
for every $\tilde p\in S\cap B_r(p)$. Thus the integral curve of $Z$ belongs to $S$. Therefore, by \eqref{tildequgualeexp} and choosing $\varrho>0$ small enough, we conclude that $\tilde q\in S$, whence the thesis follows.
\end{proof}
Ultimately, combining \Cref{counterexchnnothpcodimalta} and \Cref{sufficiencysec3highcodim}, we obtain the following characterization in arbitrary codimension. \begin{theorem}
    Let $\G$ be a Carnot group of rank $m$. Assume that $1\leq k\leq m-2.$ Then $\G$ is $k$-hypergenerated if and only if every $k$-codimensional surface with locally constant horizontal tangent space is locally a vertical $k$-codimensional plane.
\end{theorem}
\section{Non-characteristic surfaces}\label{sectionnonchar}

In this section we conclude the proof of \Cref{thm:main-theorem}. 
Indeed, we prove the following statement.
 \begin{theorem}
\label{thm:hypersurfaces}
 
     Let $\G$ be a Carnot group. The following are equivalent.
     \begin{enumerate}
         \item $\G$ is $k$-hypergenerated.
         \item On all embedded non-characteristic surfaces $S\subseteq \G$ of codimension at least $k$, the intrinsic distance and the restricted distance (as in \Cref{def:distances}) are locally bi-Lipschitz equivalent.
     \end{enumerate}
   
 \end{theorem}

The implication (2)$\Rightarrow$(1) in \cref{thm:hypersurfaces} is a straightforward consequence of the following proposition.

\begin{proposition}
\label{prop:non-hyper-gen-hyperplane}
     Let $\G$ be a Carnot group. The following are equivalent.
    \begin{itemize}
        \item $\G$ is $k$-hypergenerated.
        \item All vertical planes $H\subseteq \G$ of codimension $k$ are Carnot subgroups of $\G$.
    \end{itemize}
    In particular, if $\G$ is not $k$-hypergenerated, there exists a vertical $k$-codimensional plane $H\subseteq\G$ whose intrinsic distance is not finite.
\end{proposition}
\begin{proof}
Let $V_1\oplus \ldots \oplus V_s$ be the stratification of the Lie algebra of $\G$ and $\Delta$ be the left-invariant extension of $V_1$. Let $H\subseteq\G$ be a vertical plane of codimension $k$. Let us denote by $\mathfrak{h}$ its Lie algebra. We have
    \begin{equation}\label{eq:lieofh}
    \mathfrak{h}=P_H\oplus V_2\oplus\cdots\oplus V_s,
    \end{equation}
where $P_H:=\mathfrak{h}\cap V_1$ (cf. \eqref{verticalplanesdefeq}).
The subgroup $H$ is a Carnot subgroup if and only if $P_H\oplus V_2\oplus\cdots\oplus V_s$  is a stratification of $\mathfrak{h}$, that is, if and only if
\begin{equation*}   \Lie(P_H)=P_H\oplus[\mathfrak{g},\mathfrak{g}]\stackrel{\eqref{eq:lieofh}}{=}\mathfrak{h}.
\end{equation*}
Since the latter equation holds for every $P_H\subseteq V_1$ of co-dimension smaller than $k$ if and only if $\G$ is $k$-hypergenerated, we proved the first part of the proposition.  

Assume now that $\G$ is not $k$-hypergenerated. Then, there exists a $k$-codimensional subspace $P\subseteq V_1$ such that 
   \begin{equation}
   \label{eq:not_bracket_gen_in_hyperplane}
      [\g,\g]\not\subseteq \Lie(P). 
   \end{equation}
    Define
    $$H:=\exp\left(P\oplus [\g,\g]\right).$$
    For all $g\in H$ we have that $T_gH\cap \Delta= \mathrm{d}L_g P$. Consequently, by \eqref{eq:not_bracket_gen_in_hyperplane}, the Lie algebra generated by the distribution $TH\cap \Delta\subseteq TH$ is strictly contained in $TH$, thus by Frobenius Theorem the intrinsic distance is not finite.
\end{proof}

\subsection{Weak tangents of Lipschitz maps}

To prove the implication (1)$\Rightarrow$(2) in \cref{thm:hypersurfaces} we show that, given a hypersurface equipped with the intrinsic distance, all the weak tangents of the inclusion map are isometric embeddings. First, we recall some basic concepts in metric geometry. 

We denote with $\omega$ an arbitrary non-principal ultrafilter on the set of natural numbers \cite{drutu-kapovich}. For each sequence $(t_i)_{i\in\N}$ in a compact metric space, we use $\omega$ to choose a point $\lim_\omega(t_i)_{i\in\N}$ among the accumulation points of $(t_i)_{i\in\N}$.

\begin{definition}
    The \emph{limit of a sequence of pointed metric spaces} $(X_i,d_i,x_i)_{i\in \N}$ is the pointed metric space $(X,d,x)$ defined in the following way:
   \begin{enumerate}
       \item As a set, we consider $$X:=\{(y_i)_{i\in \N} \mid y_i\in X_i, \ 
 \sup_{i\in \N}(d_i(x_i,y_i))< +\infty )\}/\sim,$$ where $(y_i)_{i\in\N}\sim (z_i)_{i\in\N}$ if and only if $\lim_\omega d_i(z_i,y_i)=0$. 
       \item As distance, we take $d( (y_i)_{i\in\N}, (z_i)_{i\in\N})=\lim_\omega d_i(z_i,y_i)$.
       \item As base point we consider $x:=(x_i)_{i\in\N}$.
   \end{enumerate}
\end{definition}

Each sequence of $L$-Lipschitz maps $f_i:(X_i,d_i,x_i)\to (Y_i,d_i,y_i)$, with $L\in \mathbb{R}$ and where here and in the following we adopt the same name for possibly different distances, induces in a natural way a \emph{limit map} $f:(X,d,x)\to (Y,d,y)$ between the limit spaces, defined by setting $f((z_i)_{i\in \N}):=(f_i(z_i))_{i\in N}$ for all $(z_i)_{i\in \N}\in X$, which is well-defined and $L$-Lipschitz.

\begin{definition}
    Let $(X,d)$ be a metric space. We say that a pointed metric space $(Y,d,y)$ is a \emph{weak tangent} of $X$ at $x$ if there exists sequence $(\epsilon_n,x_n)_{n\in \N}\in (\R\times X)^\N$, with $\lim_{n\to\infty}\epsilon_n=0$ and $\lim_{n\to\infty} x_n=x$, such that the limit of $(X,\frac{1}{\epsilon_n}d,x_n)$ is isometric to $(Y,d,y)$.
    We will write $(X,d,x)_{(\epsilon_n,x_n)_{n\in \N}}$ to denote the weak tangent obtained via the sequence $(\epsilon_n,x_n)_{n\in \N}\in (\R\times X)^\N$.
    A weak tangent of a $L$-Lipschitz map $f:(X,d)\to (Y,d)$, is the limit $$f_{(\epsilon_n,x_n)_{n\in \N}}:(X,d,x)_{(\epsilon_n,x_n)_{n\in \N}}\to (Y,d,f(x))_{(\epsilon_n,f(x_n))_{n\in \N}}$$ 
    of the sequence of maps
    $$f:(X,\frac{1}{\epsilon_n}d,x_n)\to (Y,\frac{1}{\epsilon_n}d,f(x_n)),  $$
    for some $(\epsilon_n,x_n)_{n\in \N}\in (\R\times X)^\N$, with $\lim_{n\to\infty}\epsilon_n=0$ and $\lim_{n\to\infty} x_n=x$.
    
\end{definition}
We state a key result on weak tangents of maps that we will use in the proof \cref{thm:hypersurfaces}. 
\begin{lemma}
\label{lem:weak_tangent_bilip}
Let $f:(X,d)\to (Y,d)$ be an $L$-Lipschitz map between metric spaces. Assume that there exists $\rho>0$ such that, for all $x\in X$, every weak tangent $f_{(\epsilon_n,x_n)_{n\in \N}}$ of $f$ at $x$ satisfies
\begin{equation}
\label{eq:blow_up_is_rho_bilip}
    d(f_{(\epsilon_n,x_n)_{n\in \N}}(z) ,f_{(\epsilon_n,x_n)_{n\in \N}}(w))\geq \rho d(z,w),\quad \text{for all } z,w\in X_{(\epsilon_n,x_n)_{n\in \N}}.
\end{equation}
Then, for all $\bar{\rho}< \rho$, the map $f$ is a local $\left(\max\left\{\frac{1}{\bar{\rho}},L\right\} \right)$-bi-Lipschitz embedding.
\end{lemma}
\begin{proof}
Assume by contradiction that the lemma does not hold. Then there exists $\bar\rho<\rho$, $x\in X$ and sequences $(x_n)_{n\in\N},(y_n)_{n\in\N}$ converging to $x$ such that
\begin{equation}
\label{eq:inequality_contradiction}
    d(f(x_n),f(y_n))\leq \bar{\rho}d(x_n,y_n).
\end{equation}
Set $\epsilon_n:=d(x_n,y_n)$. We have 
\begin{eqnarray*}
    d(f_{(\epsilon_n,x_n)_{n\in \N}}((x_n)_{n\in \N}),f_{(\epsilon_n,x_n)_{n\in \N}}((y_n)_{n\in \N}))&=& \lim_\omega \frac{1}{\epsilon_n}d(f(x_n),f(y_n))\\
    &\stackrel{\eqref{eq:inequality_contradiction}}{\leq}& \lim_\omega \frac{\bar{\rho}}{\epsilon_n}d(x_n,y_n)\\
    &=&\bar{\rho}d((x_n)_{n\in \N},(y_n)_{n\in \N} ),
\end{eqnarray*}
which contradicts \eqref{eq:blow_up_is_rho_bilip}.
\end{proof}

\subsection{Weak tangents of sub-Riemannian manifolds}
In this section, we describe the weak tangents of equiregular sub-Riemannian manifolds. We follow the ideas in \cite{Bellaiche}, \cite{Margulis-Mostow}, \cite{mitchell} and \cite{gioacchino-sebastiano-enrico}, 
but in addition we allow the base point of dilations to vary. Given a point in an equiregular sub-Riemannian manifold, we prove that all weak tangents at that point are isometric to a Carnot group, which might depend on the point. We do this by describing the metric on the weak tangent as a uniform limit of sub-Riemannian metrics.

\vspace{2mm}

Let $M$ be an equiregular sub-Riemannian manifold of step $s$. Fix a stratified basis $X_1,\ldots ,X_n$ and denote with $w_i$ the weight of $X_i$ (see \cref{def:stratified_basis}). Fix $p\in M$.
We define \emph{exponential coordinates of the second kind} $\psi:\mathcal O\to M$ at $p$ by setting  
\begin{equation}
    \psi(x_1,\ldots ,x_n):=\phi_{X_1}^{x_1}\circ\ldots \circ\phi_{X_n}^{x_n}(p),
\end{equation}
for all $x_1,\ldots ,x_n\in \mathcal{O}$,  where $\mathcal{O}\subseteq \mathbb{R}^n$ is a suitable neighborhood of $0\in \R^n$ and $\phi_X^t$ denotes the flow at time $t$ along the vector field $X$. Moreover, we fix a neighborhood $U$ of $p\in M$ and an open set $A\subseteq U\times \mathbb{R}^n$ such that, for every $q\in U$, exponential coordinates of second type centered at $q$ are a diffeomorphism from $A\cap \left(\{q\}\times\mathbb{R}^n\right)$ onto $U$.
Define $\delta_\epsilon:\mathbb{R}^n\to \mathbb{R}^n$ by setting
\begin{equation*}
    \delta_\epsilon(x_1,\ldots ,x_n):=(\epsilon^{w_1}x_1,\ldots ,\epsilon^{w_n}x_n),\quad \text{ for all } x_1,\ldots ,x_n\in\mathbb{R}.
\end{equation*}
Set $A_\epsilon:=\{(q,x)\in U\times \R^n \mid (q,\delta_\epsilon(x))\in A\} \subseteq U\times \R^n$.
For $i\in\{1,\ldots ,n\}$, define $X_i^\epsilon:A_\epsilon\to \R^n$ by setting
\begin{equation}
\label{eq:def_X_eps}
    X_i^\epsilon(q,x):=\epsilon^{w_i} \left(\delta_{\frac{1}{\epsilon}}\right)^*\psi_q^*X_i(x).
\end{equation}

The following proposition ensures that the above vector fields converge locally uniformly as $\epsilon\to 0$.

\begin{proposition}
\label{prop:convergence_of_vector_fields}
For all $i\in \{1,\ldots ,n\}$, there exists $X_i^0:U\times \mathbb{R}^n \to \R^n$ such that, for all compact neighborhoods $H\subseteq \R^n$ of $0$ and for all compact sets $K\subseteq U$, there exists $\bar{\epsilon},C>0$ such that for all $\epsilon\leq \bar{\epsilon}$ we have $K\times H \subseteq A_\epsilon $ and
    \begin{equation*}
        \|X_i^\epsilon(q,x)-X_i^0(q',x)\|\leq C(d(q,q')+\epsilon), \quad \text{for all $x\in H$, for all $q,q'\in K$.}
    \end{equation*} 
    Moreover, for all $q\in U$, the vector fields $x\mapsto X_i^0(q,x)$, with $i\in\{1,\ldots ,n\}$, form a stratified algebra.
\end{proposition}

\begin{proof}
    Let $P^k:A\to \R$ be the coefficients of $\psi_q^*X_i$:
    \begin{equation*}
        \psi_q^*X_i(x)=\sum_{k=1}^n P^k(q,x)\partial_k,\quad \text{for all } x\in A. 
    \end{equation*}
Fix $N:=\max\{w_1,\ldots ,w_n\}$.
Since the $P^k$'s are smooth, we can expand them as a Taylor series in a neighborhood of the origin
\begin{equation}
  \psi_q^*X_i(x)=\sum_{k=1}^n\left(\left(\sum_{\substack{\alpha\in \N^n \\ |\alpha|\leq N} \\
  } a_\alpha^k(q)x^\alpha \right)+ o(\|x\|^N)\right)\partial_k. 
\end{equation}
where for a multi-index $\alpha=(\alpha_1,\ldots ,\alpha_n)$ we denote $x^\alpha:=x_1^{\alpha_1}\cdot\ldots\cdot x_n^{\alpha_n} $, $|\alpha|:=\alpha_1+\cdots+\alpha_n$, and $a_\alpha^k:U\to \R$ are suitable smooth functions.
Denote $[\alpha]:=\sum_1^n \alpha_iw_i$.
By Bella\"{\i}che work \cite{Bellaiche}
, we have that $a^k_\alpha(q)=0 $ if $[\alpha]< w_k-1$. 
Define $X_i^0:U\times \mathbb{R}^n \to \R^n$ by setting
\begin{equation*}
    X_i^0(x)=\sum_{k=1}^n\left(\sum_{\substack{\alpha\in \N^n \\ [\alpha]=w_k-1}
  } a_\alpha^k(q)x^\alpha \right)\partial_k,\quad \text{for all } x\in A.
\end{equation*}
Set
$$C_1:=2\sum_{\substack{k\in\{1,\ldots ,n\},\\ \alpha\in \N^n,\\ |\alpha|\leq N}}\max\left\{a_\alpha^k(q)x^\alpha \mid  q\in K, x\in H \right\}$$
and
$$C_2:=2\max\{C_1,\mathrm{Lip}_{K\times H}(X_i^0)\},$$
where $\mathrm{Lip}_{K\times H}(X_i^0)$ denotes the Lipschitz constant of the smooth function $X_i^0$ on $H\times K$ with respect to $d\times \|\cdot\|$.
Since $x^\alpha\circ \delta_\epsilon=\epsilon^{[\alpha]}x^\alpha$ and $\left(\delta_{\epsilon}\right)^*\partial_k=\epsilon^{-w_k}\partial_k$, we have 

\begin{eqnarray*}
    X_i^\epsilon(q,x)= \sum_{k=1}^n\left(\left(\sum_{\substack{\alpha\in \N^n \\ |\alpha|\leq N}
  } \epsilon^{1+[\alpha]-w_k} a_\alpha^k(q)x^\alpha \right)+ o(\epsilon^{N-w_k+1})\right)\partial_k. 
\end{eqnarray*}
Consequently, we can choose $\bar{\epsilon}$ such that
\begin{equation*}
   \left\|X_i^\epsilon(q,x)-X_i^0(q,x) \right\|= \left\|\sum_{k=1}^n\left(\left(\sum_{\substack{\alpha\in \N^n \\ |\alpha|\leq N\\ [\alpha]>w_k-1} \\
  } \epsilon^{1+[\alpha]-w_k} a_\alpha^k(q)x^\alpha \right)+ o(\epsilon^{N-w_k+1})\right)\partial_k\right\|\leq C_1\epsilon,
\end{equation*}
for all $x\in H$ and $q\in K$.
Therefore,
\begin{eqnarray*}
    \left\|X_i^\epsilon(q,x)-X_i^0(q',x) \right\|\leq \left\|X_i^\epsilon(q,x)-X_i^0(q,x) \right\|+\left\|X_i^0(q,x)-X_i^0(q',x) \right\|\leq C_2(d(q,q')+\epsilon).
\end{eqnarray*}
For the proof that for all $q\in U$, the vector fields $x\mapsto X_i^0(q,x)$, with $i\in\{1,\ldots ,n\}$, form a stratified algebra, we refer to \cite[Propositions 5.17-5.22]{Bellaiche}.
\end{proof}
Continuing with the terminology for the setting discussed in this section, fix two compact sets $H\subseteq\R^n$, $K\subseteq U$ and $\bar{\epsilon}$ such that $ K\times H\subseteq A_\epsilon$ for all $\epsilon\leq \bar{\epsilon}$. Fix $\epsilon\leq \bar{\epsilon}$ and $q\in K$, denote with $d_{\epsilon,q}$ and $d_{0,q}$ the sub-Riemannian distances on $H$ induced respectively by the sub-Riemannainan metrics for which  $\{X_i^\epsilon(q,\cdot)\}_{i\in\{1,\ldots ,n\}, w_i=1} $ and $\{X_i(q,\cdot)^0\}_{i\in\{1,\ldots ,n\}, w_i=1} $ are orthonormal frames respectively.
An immediate consequence of \cref{prop:convergence_of_vector_fields} and \cite[Theorem C.2]{gioacchino-sebastiano-enrico} is the following result.

\begin{corollary}
\label{cor:convergence_distances}
    For all sequences $\epsilon_m\to 0$ and $q_m\to q$, the distances $d_{\epsilon_m,q_m}$ converge uniformly on compact sets to $d_{0,q}$.
\end{corollary}

As a consequence of \cref{cor:convergence_distances}, we get the following.

\begin{corollary}
\label{cor:weak-tangent-of-equiregular-sR-abstract}
At every point of an equiregular sub-Riemannian manifold, there is a unique weak tangent that is isometric to a Carnot group.
\end{corollary}

To prove \cref{cor:weak-tangent-of-equiregular-sR-abstract} we state and prove a slightly stronger statement. Indeed, using the terminology introduced in this section, we can explicitly describe the weak tangent of equiregular sub-Riemannian manifolds.

\begin{corollary}
\label{cor:weak-tangent-of-equiregular-sR}
    Fix $q\in U$. Every weak tangent of $M$ at $q$ is isometric to $(\R^n,d_{0,q},0)$.
\end{corollary}

\begin{proof}
Fix a sequence $(\epsilon_m,q_m)_{m\in \N}\in (\R\times X)^\N$, with $\lim_{m\to\infty}\epsilon_m=0$ and $\lim_{m\to\infty} q_m=q$. Fix $R>0$ and denote with $B:=B_R(0)$ the ball centered at $0$ of radius $R$ with respect to the distance $d_{\epsilon_m,q_m}$.
Fix $N\in\N$ and a compact neighborhood $K$ of $q$ such that for all $m\geq N$  we have $ K\times B \subseteq A_{\epsilon_m} $ and $q_m\in K$.
Define the map $\varphi_m: (B,d_{\epsilon_m,q_m},0)\to (M,\frac{1}{\epsilon_m}d, q_m)$ by setting
\begin{eqnarray}
\label{eq:def_varphin}
    \varphi_m:=\psi_{q_m}\circ \delta_{\epsilon_m}.
\end{eqnarray}
We claim that $\varphi_m$ is an isometry of pointed metric spaces. Indeed, clearly $\varphi_m(0)=q_m$. Moreover, $\frac{1}{\epsilon_m}d$ is the sub-Riemannian distance induced by $\{ \epsilon_m X_i \}_{i\in\{1,\ldots ,n\}, w_i=1}$. Thus, since the orthonormal frame $\{X_i^{\epsilon_m}(q_m,\cdot)\}_{i\in\{1,\ldots ,n\}, w_i=1} $ for the distance $d_{\epsilon,q_m}$ is exactly $\{\varphi_m^*\epsilon_m X_i \}_{i\in\{1,\ldots ,n\}, w_i=1}$ (see \eqref{eq:def_X_eps}), we have that $\varphi_m$ is an isometric embedding. Finally, being $\varphi_m$ an isometric embedding and being the map $\psi_{q_m}$ surjective onto $U$, we must have that $\varphi_m$ is surjective onto $B(q_m,\epsilon_mR)$.
Hence the two weak tangents $(B_R(0), d_{0,q},0)$ and $(B_R,d,q)_{(\epsilon_m,q_m)_{m\in \N}}$ are isometric. 
Since the weak tangent of a sequence of isometries is an isometric embedding, and by \cref{cor:convergence_distances} the pointed metric spaces $(B,d_{\epsilon,q_m},0)$ converge to $(B_R(0), d_{0,q},0)$, the limit map of the sequence $\varphi_m$ is an isometric embedding
\begin{equation*}
    \varphi:(B_R(0), d_{0,q},0)\to (B_R,d,q)_{(\epsilon_m,q_m)_{m\in \N}}\simeq (B_R(0), d_{0,q},0),
\end{equation*}
where $B$ is the ball of radius $R$ in $(M,d,q)_{(\epsilon_m,q_m)_{m\in \N}}$. Being $\varphi$ an self isometric embedding of a proper metric space it is an isometry. Since $R>0$ was arbitrary, the corollary is proved.\end{proof}

We now consider maps between $M$ and a second equiregular sub-Riemannian manifold $M'$. To smoothen the notation, we will denote with an apex $'$ all the object in $M'$.

\begin{remark}
\label{rem:blow-up-maps}
In the proof of \cref{cor:weak-tangent-of-equiregular-sR}, we explicitly write the isometry between the weak tangent $(M,d,q)_{(\epsilon_m,q_m)_{m\in \N}}$ and $(\R^n,d_{0,q},0)$ as the limit of the maps in \eqref{eq:def_varphin}.
Let $f:M\to M'$ be a map between sub-Riemannian manifolds, and $\varphi'_m:= \psi_{f(q_m)}\circ\delta_{\epsilon_m}$.
Denote with $\varphi, \varphi'$ the limit maps of $\varphi_m$ and $\varphi'_m$ respectively.
If $f_m:= (\varphi '_m)^{-1}\circ f\circ \varphi_m= \delta_{\frac{1}{\epsilon_m}}\circ \psi_{f(q_m)}^{-1}\circ f\circ \psi_{q_m}\circ \delta_{\epsilon_m}$ converges uniformly on compact sets to a map $f_{\infty}:\R^{n}\to \R^{n'}$,
then the following diagram commutes:

\begin{equation*}
\begin{tikzcd}
{(M,d,q)_{(\epsilon_m,q_m)_{m\in \N}}} \arrow[rr, "f_{(\epsilon_m,q_m)_{n\in \N}}"]                 &  & {(M',d,f(q))_{(\epsilon_m,f(q_m))_{n\in \N}}}     \\
{(\R^{n},d_{0,q},0)} \arrow[rr, "f_\infty"] \arrow[u, "\varphi"] &  & {(\R^{n'},d_{0,f(p)},0)} \arrow[u, "\varphi'"]
\end{tikzcd}
\end{equation*}
In this case, since the maps $\varphi$ and $\varphi'$ are isometries, with a slight abuse of notation we will identify $f_{(\epsilon_m,q_m)_{m\in \N}}$ with
$f_{\infty}:\R^{n}\to \R^{n'}$ and say that $f_\infty$ is a weak tangent of $f$ at $q$. 
\end{remark}

Using \cref{rem:blow-up-maps}, we prove that contact maps between equiregular sub-Riemannian manifolds admit a unique weak tangent.

\begin{definition}
A map $f:M\to M'$ between sub-Riemannian manifolds is a \emph{contact map} if it is smooth and for all $p\in M$ there holds 
\begin{equation}
\label{eq:being_contact}
    \df f_p \left(\Delta_p\right)\subseteq \Delta'_{f(p)}.
\end{equation}
\end{definition}

Before studying the weak tangents of contact maps, we state a useful technical lemma.
\begin{lemma}
\label{lem:uniform_boundedness}
    Let $g_m:\R^n\to \R^{n'}$ a sequence of continuously differentiable maps and $(\epsilon_m,q_m)_{m\in \N}\in (\R\times X)^\N$. Assume that the sequence $q_m$ converges to $q\in \R^n$.  If on every compact set $\df g_m X_i^{\epsilon_m}(q_m,\cdot)$ is uniformly bounded for all $i\in\{1,\ldots ,n\}$, then $\df g_m$ is bounded on compact sets.
\end{lemma}

The following proposition shows that the weak tangent of contact maps at a point is unique and it is a group homomorphism.


\begin{proposition}
\label{prop:weak_contact}

Let $f:M\to M'$ to be a contact map between equiregular sub-Riemannian manifolds. At every point $p\in M$, the map $f$ has a unique weak tangent $P_{f,p}:(\R^{n},d_{0,p},0)\to (\R^{n'},d_{0,f(p)},0) $, that is the unique group homomorphism satisfying 

\begin{equation}
\label{eq:formula_for_Pansu_differential}
  \df (P_{f,p})_0(X_i^0(0)) = \pi_{w_i}((\df \tilde{f_p})_0(X_i^1(0))),\quad \text{ for all } i\in\{1,\ldots ,n\}.  
\end{equation}
where $\tilde{f_q}:= \psi_{f(q)}^{-1}\circ f\circ \psi_q$ for all $q\in M$, and $\pi_{w_i}$ is the projection to the homogeneous component of weight $w_i$.
In particular, when $M'$ is a Carnot group, we have
\begin{equation}
\label{eq:formula-pansu-target-Carnot}
    \df (P_{f,p})_0(X_i^0(0)) = \pi_{w_i}((\df f)_p X_i(p)),\quad \text{ for all } i\in\{1,\ldots ,n\}. 
\end{equation}
\end{proposition}

\begin{proof}
    By \cref{rem:blow-up-maps} it is enough to show that $f_{p,\epsilon}:=\delta_{\frac{1}{\epsilon}}\circ \psi_{f(q)}^{-1}\circ f\circ \psi_q\circ \delta_\epsilon: \R^n\to \R^{n'}$ converges uniformly on compact sets to a group homomorphism $P_{f,p}:\R^{n}\to\R^{n'}$ satisfying \eqref{eq:formula_for_Pansu_differential} as $\epsilon$ tends to $0$ and $q$ tends to $p$.
Fix two compact sets $H\subseteq \R^n$ and $K\subseteq U$, with $p\in K$. By \eqref{eq:being_contact}, for all $i\in\{1,\ldots ,n\}$, we have
\begin{equation*}
    \df f_qX_i\in (\Delta')^{(w_i)}_{f(q)}, \quad \text{for all } q\in M.
\end{equation*}
Consequently, we can write $\df \tilde f_{q}(X_i^1)$, where $X_i^1$ is $X_i^\epsilon$ for $\epsilon=1$, as a combination of vector fields of weight smaller than $w_i$:
\begin{equation}
\label{eq:primo_conto_differenziale_f-tilde}
    \df \tilde f_{q}(X_i^1(q,x))=\sum_{\substack{k\in\{1,\ldots ,n'\}\\ w'_k\leq w_i}} a_i^k(q,x)X'_k(f(q),\tilde{f}(x))\quad\text{for all } (q,x)\in A,
\end{equation}
    for some smooth functions $a_i^1,\ldots ,a_i^k: A\to \R$, for all $i\in \{1,\ldots ,n\}$.
From the above equation we get

\begin{eqnarray}
    \label{eq:dove_va_X-i}
\df \delta_{\frac{1}{\epsilon}}\df \tilde{f}_q \df \delta_\epsilon(X_i^\epsilon(q,x))&\stackrel{\eqref{eq:def_X_eps}}{=}&  \epsilon^{w_i} \df \delta_{\frac{1}{\epsilon}}\df \tilde{f}_q (X_i^1(q,\delta_\epsilon(x)))\nonumber \\ 
&\stackrel{\eqref{eq:primo_conto_differenziale_f-tilde}}{=}& \epsilon^{w_i} \df  \sum_{\substack{k\in\{1,\ldots ,n'\}\\ w'_k\leq w_i}} a_i^k(q,\delta_\epsilon(x)) \df\delta_{\frac{1}{\epsilon}}X'_k(f(q),\tilde{f}(\delta_\epsilon(x)))\\ \nonumber
&=&    \sum_{\substack{k\in\{1,\ldots ,n'\}\\ w'_k\leq w_i}}\epsilon^{w_i-w'_k} a_i^k(q,\delta_\epsilon(x)) (X'_k)^\epsilon(f(q),\delta_{\frac{1}{\epsilon}}\circ\tilde{f}\circ\delta_\epsilon(x)).
\end{eqnarray}
Fix now a sequence $(\epsilon_m,q_m)_{m\in \N}\in (\R\times X)^\N$, with $\lim_{m\to \infty}\epsilon_m=0$ and $\lim_{m\to\infty} q_m=p$, such that $f_{q,\epsilon}$ converges uniformly on compact sets to some $f_\infty$. 
The right-hand side of equation \eqref{eq:dove_va_X-i} converges to
\begin{equation*}
Y_i(q,x):=\sum_{\substack{k\in\{1,\ldots ,n'\}\\ w'_k = w_i}} a_i^k(q,0)(X'_k)^0(f(q),f_\infty(x))
\end{equation*}
 as $n$ tends to $\infty$. As a consequence, by \cref{lem:uniform_boundedness} $\df f_{q_m,\epsilon_m}$ is uniformly bounded.
 Thus, by \cref{prop:convergence_of_vector_fields} we have
 $\lim_{m\to\infty} \df f_{q_m,\epsilon_m}(X_i^\epsilon(q_m,x)-X_i^0(q,x))=0$. Consequently,

 \begin{eqnarray*}
     \| \df f_{q_m,\epsilon_m}X_i^0(x) -Y_i(p,x)\|&\leq & 
 \| \df f_{q_m,\epsilon_m}X_i^\epsilon(x) - \df f_{q_m,\epsilon_m}X_i^0(x)\|+\|\df f_{q_m,\epsilon_m}(X_i^\epsilon(q_m,x))-Y_i(p,x))\| \\
 &\stackrel{m\to\infty}{\to}& 0
 \end{eqnarray*}
Since
\begin{equation*}
    Y_i(p,0)=\pi_{w_i}((\df \tilde{f_q})_0(X_i^1(0))),
\end{equation*}
we proved that $\df f_{q_m,\epsilon_m}$ converges uniformly to a map satisfying \eqref{eq:formula_for_Pansu_differential}, thus $f_\infty$ is differentiable and it is the unique map such that
\begin{equation*}
    \df f_{\infty}X_i^0(x)=Y_i(p,x).
\end{equation*}
We finally remark that the vector fields $Y_i(p,\cdot)$'s are a linear combination of the $X'_i$'s and therefore they are left-invariant vector fields for the group structure on $\R^{n'}$. It is an easy exercise to check that a smooth map between Lie groups whose differential sends left-invariant vector fields to left-invariant vector fields is a group homomorphism. 
When $M'$ is a Carnot group, we can identify $\R^{n'}$ with $M'$ itself. We have then that $\df \psi'_{f(p)}=\id$. Hence, we get \eqref{eq:formula-pansu-target-Carnot} by \eqref{eq:formula_for_Pansu_differential} observing that $\df \tilde{f}_0(X_i^1)=\df f_p X_i(p) $
\end{proof}

\begin{definition}
For $k,m\in \mathbb{N}$ define $\mathcal{I}_m^k:=\{(i_1,\ldots ,i_k)\in \mathbb{N}^k \mid i_1,\ldots ,i_k\in \{1,\ldots ,m\} \}$.
 Let $M$ be a smooth manifold and let $X_1,\ldots ,X_m\in\Gamma(TM)$. For $k\in\N$ and for a multi-index $J=(j_1,\ldots ,j_k)\in \mathcal{I}_m^k$, we define the {\em iterated bracket} associated to $J$ as
	\begin{equation}
 \label{eq:iterated_bracket}
		X_J:=[X_{j_1},[ \ldots ,[  X_{j_{k-1} }, X_{j_k}]\ldots ]].
	\end{equation} 
\end{definition}



%

We are ready to prove the main result of this section.
\begin{proof}[Proof of \Cref{thm:hypersurfaces}]
Since the implication (2)$ \Rightarrow $(1) follows from \Cref{prop:non-hyper-gen-hyperplane}, we are left to prove (1)$\Rightarrow$(2). 
Let $\G$ be a $k$-hypergenerated group. Notice that $\G$ is $k'$-hypergenerated for every $k'\leq k$, thus we do not lose generality if we prove that the intrinsic and the restricted distance are bi-Lipschitz equivalent on surfaces that have codimension exactly $k$.
Let $S\subseteq \G$ be an embedded surface of codimension $k$. 
We start proving that $\Delta_S:=TS\cap \Delta$ is bracket-generating inside $TS$, and therefore that the intrinsic distance $d_i$ is finite. Fix $p\in S$. Up to performing a left-translation we can assume without loss of generality that $p=\id$.
Since $S$ is non-characteristic we can choose a left-invariant frame $X_1,\ldots ,X_{m}$ of $\Delta$ such that $$\spann\{X_1(\id),\ldots ,X_r(\id)\}=T_{\id}S\cap \Delta_{\id},$$ with $r:=m-k$. Being $S$ a smooth surface, there exists a neighborhood $U\subseteq S$ of $p$ and smooth functions $a_i^j:U\to \mathbb{R}$, with $i\in \{1,\ldots ,r\}, j\in \{r+1,\ldots ,m\}$, such that, if 
\begin{equation}
    \label{eq:form_Y_i-first}
    Y_i:=X_i+\sum_{j=r+1}^m a_i^jX_j,
\end{equation} 
for all $i\in\{1,\ldots ,r\}$, then $Y_1,\ldots ,Y_r$ form a frame for the distribution $\Delta_S$ on $U$.
Notice that for all $i\in \{1,\ldots ,r\}, j\in \{r+1,\ldots ,m\}$, we have $a_i^j(\id)=0$. We claim that for every multi-index $J\in \mathcal{I}_{r}^h$ there exist $n_J\in\mathbb{N}$, $R_J,T^1_J,\ldots T_J^{n_J}\in\Gamma(TU)$  and $\alpha^1_J,\ldots ,\alpha_J^{n_J}:U\to \mathbb{R}$ smooth such that
\begin{eqnarray}
\label{eq:form_Yj_first}
    Y_J&=&X_J+R_J+\sum_{i=1}^{n_J}\alpha^i_J T^i_J,
 \\
\label{eq:form_Yj_second}
    R_J(q)&\in& \df L_q (V_1\oplus\ldots \oplus V_{h-1})\quad \text{ for all } q\in U,\\
\label{eq:form_Yj_third}
    T^i_J(q)&\in& \df L_q (V_1\oplus\ldots \oplus V_{h}) \quad \text{ for all } q\in U, \text{ for all } i\in\{1,\ldots ,n_J\},\\
\label{eq:form_Yj_fourth}
    \alpha^i_J(\id)&=&0\quad \text{ for all } i\in\{1,\ldots ,n_J\},
\end{eqnarray}
where $Y_J$ and $X_J$ are the iterated brackets of the $Y_j's$ and $X_j's$ respectively (see \eqref{eq:iterated_bracket}).
We prove the claim by induction over $h$.
For $j\in \{1,\ldots ,r\}$ we can choose $R_j:=0$, $n_j:=m-r$, $\alpha_j^i:=a_{j}^{r+i}$, $T_j^i:=X_{r+i}$. Thus the claim for $h=1$ is proved.
Assume now that the statement is true for some $h$ and let $J=(j_1,\ldots ,j_{h+1})\in\mathcal{I}_{r}^{h+1}$. Set $J':=(j_2,\ldots ,j_{h+1})$. By inductive hypotheses we can write
\begin{equation}
\label{eq:to_prove_form_Y}
Y_J=[Y_{j_1},Y_{J'}]=[Y_{j_1},X_{J'}+R_{J'}+\sum_{i=1}^{n_{J'}}\alpha_{J'}^iT_{J'}^i],
\end{equation}
for some $n_{J'},R_{J'},T_{J'}^1,\ldots ,T_{J'}^{n_{J'}},\alpha_{J'},\ldots ,\alpha_{J'}^{n_{J'}}$ for which \eqref{eq:form_Yj_first},\eqref{eq:form_Yj_second},\eqref{eq:form_Yj_third}, and \eqref{eq:form_Yj_fourth} hold. Thus
\begin{eqnarray*}
    Y_J&\stackrel{\eqref{eq:to_prove_form_Y}}{=}&[Y_{j_1},X_{J'}+R_{J'}+\sum_{i=1}^{n_{J'}}\alpha^i_{J'} T^i_{J'}]
    \\
    &=& [ Y_{j_1},X_{J'}]+[Y_{j_1},R_{J'}]+\sum_{i=1}^{n_{J'}}\left( Y_{j_1}\alpha_{J'}^i \right) T^i_{J'}+\sum_{i=1}^{n_{J'}}\alpha^i_{J'} [Y_{j_1},T^i_{J'}]\\
    &\stackrel{\eqref{eq:form_Y_i-first}}{=}& 
X_J+\sum_{l=r+1}^{m} a_{j_1}^l[X_l,X_{J'}] -\sum_{l=r+1}^{m} \left(X_{J'}a_{j_1}^l\right)X_{l}+[Y_{j_1},R_{J'}]\\ && +\sum_{i=1}^{n_{J'}}\left( Y_{j_1}\alpha_{J'}^i \right) T^i_{J'}+
\sum_{i=1}^{n_{J'}}\alpha^i_{J'}[X_{j_1},T^i_{J'}]+\sum_{l=r+1}^m\sum_{i=1}^{n_{J'}}a_{j_1}^l\alpha^i_{J'} [X_{l},T^i_{J'}],\\
&& -\sum_{l=r+1}^m\sum_{i=1}^{n_{J'}} \alpha^i_{J'}\left(T^i_{J'}a_{j_1}^l\right) X_{l}.
\end{eqnarray*}
Setting
\begin{eqnarray*}
    n_J&:=& (m-r+1)n_{J'}+m-r,\\
    R_J&:=&-\sum_{l=r+1}^{m} \left(X_{J'}a_{j_1}^l\right)X_l+[Y_{j_1},R_{J'}]+ \\ && \sum_{i=1}^{n_{J'}}\left( Y_{j_1}\alpha_{J'}^i \right) T^i_{J'}-\sum_{l=r+1}^m\sum_{i=1}^{n_{J'}} \alpha^i_{J'}\left(T^i_{J'}a_{j_1}^l\right) X_{l},\\
T_J^l&:=&[X_l,X_{J'}] ,\quad l\in\{1,\ldots ,m-r\},\\
T_J^{m-r+i}&:=& [X_{j_1},T^i_{J'}], \quad \text{ for all } i\in\{1,\ldots ,n_{J'}\}\\
T_J^{m-r+n_{J'}l+i}&:=& [X_{l+r},T^i_{J'}],\quad \text{ for all } l\in\{1,\ldots ,m-r\},i\in\{1,\ldots ,n_{J'}\}\\
    \alpha_J^l&:=&a_{j_1}^l ,\quad l\in\{1,\ldots ,m-r\},\\
\alpha_J^{m-r+i}&:=& \alpha^i_{J'}, \quad \text{ for all } i\in\{1,\ldots ,n_{J'}\}\\
\alpha_J^{m-r+n_{J'}l+i}&:=& a_{j_1}^l\alpha^i_{J'},\quad \text{ for all } l\in\{1,\ldots ,m-r\},i\in\{1,\ldots ,n_{J'}\},\\
\end{eqnarray*}
it is an exercise to check that \eqref{eq:form_Yj_first},\eqref{eq:form_Yj_second},\eqref{eq:form_Yj_third}, and \eqref{eq:form_Yj_fourth} hold, thus the proof of the claim is concluded.
For all $h\in \mathbb{N}$ we have 
\begin{equation}
\label{eq:inclusion_between_distributions}
    \left(\Delta_S^h\right)_\id\subseteq\left( \Delta^h\right)_\id\cap T_\id S.
\end{equation}
We prove by induction over $h$ that
\begin{equation}
\label{eq:claim_distributions_are_equal}    \left(\Delta_S^h\right)_\id =\left( \Delta^h\right)_\id\cap T_\id S.
\end{equation}
For $h=1$ the claim \eqref{eq:claim_distributions_are_equal} is clear being the inclusion a contact map. 
Assume that \eqref{eq:claim_distributions_are_equal} holds for some 
$h\in\mathbb{N}$ with $h\geq 1$.
By \eqref{eq:inclusion_between_distributions} it is enough to prove
\begin{equation}
\label{eq:distributions_have_equal_dimensions}
\dim\left(\left(\Delta_S^{h+1}\right)_\id \right)=\dim\left(\left( \Delta^{h+1}\right)_\id \cap T_\id S\right).
\end{equation}
For every multi-index $J\in \mathcal{I}_m^{h+1}$, by \eqref{eq:form_Yj_first} and \eqref{eq:form_Yj_fourth} we have that $Y_J(\id)=X_J(\id)+R_J(\id)$ with $R_J$ satisfying \eqref{eq:form_Yj_second}, thus 
\begin{eqnarray*}   \dim\left(\left(\Delta_S^{h+1}\right)_\id\right)&\geq \dim\left(\left(\Delta_S^{h}\right)_\id\right)+\dim(\mathrm{span}\{X_J(\id) \mid J\in\mathcal{I}_r^{h+1}\})\\ 
&= \dim\left(\left(\Delta_S^{h}\right)_\id\right)+\dim(V_{h+1})\\
&\stackrel{\eqref{eq:claim_distributions_are_equal}}{\geq}\dim\left(\left( \Delta^{h+1}\right)_\id\cap T_\id S\right),
\end{eqnarray*}
where in the first inequality we used that $R_J \in \left(\Delta_S^{h}\right)$ for all $J\in\mathcal{I}_r^{h+1} $ and in the second equality we used that $\mathrm{span}\{X_J \mid J\in\mathcal{I}_r^{h+1}\}=V_{h+1}$ since the group is $k$ generating and $r=m-k$.
Consequently, we have that \eqref{eq:distributions_have_equal_dimensions} holds and that the claim $\eqref{eq:claim_distributions_are_equal}$ is proved.
In particular, since $\Delta^s=TG$ for some $s\in\mathbb{N}$, by \eqref{eq:claim_distributions_are_equal} we have that $\Delta_S$ is bracket-generating.
Denote with $f:S\to \G$ the canonical inclusion, where $S$ is equipped with the intrinsic distance $d_i$. Since $d_i\leq d_r$ the map $f$ is $1$-Lipschitz. 
By \Cref{prop:weak_contact}, $f$ has a unique weak tangent  $P_{f,p}$ at $p\in S$, and for all $i\in\{1,\ldots ,m\}$ we have 
\begin{eqnarray}
\label{eq:conto_Pansu}
    \left(P_{f,p}\right)_*(Y_i^0)\stackrel{\eqref{eq:formula-pansu-target-Carnot}}{=}\pi_1(\df f_p Y_i(p))
    \stackrel{\eqref{eq:form_Y_i-first}}{=}&\pi_1(X_i(p))
    = X_i(p),
\end{eqnarray}
where $Y_i^0$ is the vector-field given by \cref{prop:convergence_of_vector_fields}.
Thus, being $\G$ $k$-hypergenerated, and $r=m-k$, the Lie algebra generated by the image of the $(P_{f,p})_*$ is $\Span(X_1,\ldots ,X_r)\oplus [\mathfrak{g},\mathfrak{g}]$. Being $P_{f,p}$ injective by \eqref{eq:conto_Pansu}, the Lie algebra generated by $\{Y_i^0\}_{i\in\{q,\ldots ,m\}}$ is isomorphic to $\Span(X_1,\ldots ,X_r)\oplus [\mathfrak{g},\mathfrak{g}]$ via
the map $P_{f,p}$. 
Since the map $P_{f,p}$ is an isometric embedding, by \cref{lem:weak_tangent_bilip} we get that the map $f$ is locally $C$-bi-Lipschitz for all $C>1$.
\end{proof}

\begin{remark}
    In the proof \cref{thm:hypersurfaces}, the injectivity of the Pansu differential of the inclusion is a key ingredient. An arbitrary contact injective map may not be a bi-Lipschitz embedding. For example, if $M$ is a sub-Riemannian manifold and $p\in M$, there exists a neighborhood of $p$ that can be embedded with a contact map in a Carnot group of step $2$ (see \cite[Proposition 6.5.1]{Mongomery_book}). Nontheless, if the step of $M$ is greater than $2$, this embedding is not a bi-Lipschitz map.
\end{remark}

\bibliography{hypergenerated}

\end{document}